\documentclass[a4paper]{article}
\usepackage{amsmath,amsthm,amssymb}
\usepackage{enumerate}
\usepackage[all]{xy}
\usepackage{geometry}
\usepackage[normalem]{ulem}
\usepackage{cancel}
\usepackage{color}
\usepackage{titling}
\newcommand{\antipode}{S}
\newcommand{\al}{a[l]}

\newcommand{\cat}{\mathcal{C}}
\newcommand{\coaction}{\rho}
\newcommand{\comult}{\triangle}
\newcommand{\counit}{\varepsilon}

\newcommand{\fcal}{\mathcal{F}}
\newcommand{\gl}{\mathrm{GL}}

\newcommand{\id}{\mathrm{id}}

\newcommand{\ind}[1]{\pi_{#1}^{\circ}}
\DeclareMathOperator{\Ind}{Ind}
\renewcommand{\k}{\mathbb{K}}
\newcommand{\ktr}{\k_{\mathrm{tr}}}

\newcommand{\N}{\mathbb{N}}
\newcommand{\n}{\mathbf{n}}
\newcommand{\rs}{\mathrm{\mathfrak{H}(\Sigma_n)}}
\newcommand{\hecke}[1]{\mathrm{\mathfrak{H}(\Sigma_{#1})}}

\newcommand{\ty}{\widetilde{Y}}
\newcommand{\Z}{\mathbb{Z}}

\newtheorem{theorem}{Theorem}[section]
\newtheorem{proposition}[theorem]{Proposition}
\newtheorem{proposition-definition}[theorem]{Proposition-Definition}

\newtheorem{corollary}[theorem]{Corollary}
\theoremstyle{remark}

\newtheorem{remark}[theorem]{Remark}

\theoremstyle{definition}

\newcommand{\si}[1]{\scriptsize{\emph{#1}}}

\preauthor{}
\postauthor{}
\DeclareRobustCommand{\authorthing}{
\begin{center}
	\begin{tabular*}{0.75\textwidth}{@{\extracolsep{\fill}}lc}
		 \quad Ana Paula Santana\thanks{This work was partially supported by the Centre for Mathematics of the
University of Coimbra -- UID/MAT/00324/2013, funded by the Portuguese
Government through FCT/MEC and co-funded by the European Regional
Development Fund through the Partnership Agreement PT2020.
} &
		Ivan Yudin\thanksmark{1}\thanksgap{0.3em}\thanks{The second
		author's work was  supported by  Programa
Investigador FCT IF/00016/2013.}\\	
 \si{\quad aps@mat.uc.pt} & \si{\quad yudin@mat.uc.pt}\\
\multicolumn{2}{l}{ \si{CMUC, Department of Mathematics, University of Coimbra,
Coimbra, Portugal} }\\
	\end{tabular*}

\vspace{5ex}
{\emph{In memory of J.A.~Green, with the highest admiration.}}
\end{center}
}
\date{}
\author{\authorthing}

\title{ An action of the Hecke monoid on rational modules for the
Borel subgroup of a quantised general linear group.   }

\begin{document}

\maketitle
\begin{abstract}
We construct an  action of the Hecke monoid  on the category of
rational modules
for the quantum negative Borel subgroup  of the quantum general linear group.
We also show that this action restricts to the category of polynomial modules for
this quantum subgroup and induces an action on the category of modules for
the
quantised Borel-Schur algebra~$S^-_{\alpha,\beta}(n,r)$.
\end{abstract}
\section{Introduction}

 The idea of ``\emph{categorification}'' originates from
the joint work \cite{frenkel} of Crane and Frenkel, and the term was
coined later in  Crane's article \cite{crane}.
Recently \emph{categorification} became an intensively studied subject
in several mathematical areas.  A detailed account on this topic can be found in~\cite{mazorchuk}.

Given an action   $\rho$ of a monoid  $M$ on the Grothendieck group
$\mathrm{Gr}(\cat)$ of a category $\cat$,
one can ask if it can be categorified, that is if there exists an action of $M$ on $\cat$ that induces $\rho$.
Note that, to find such an action, one has: (i)  to find a set of functors
$F_m$, $m\in M$, whose action on $\cat$ gives operators $\rho(m)$ on
$\mathrm{Gr}(\cat)$; (ii) to show the existence of a
coherent family of natural isomorphisms $\lambda_{m,m'}\colon F_m F_{m'} \to
F_{mm'}$. Usually, (ii) is much  more tricky  than (i).

Let $B_n$  be the negative Borel subgroup of the general linear group of degree
$n$  over an algebraically closed field.
Denote by $\mathrm{Gr}$ the Grothendieck group of the category of finite
dimensional polynomial $B_n$-modules. The tensor product of modules turns
$\mathrm{Gr}$ into a ring.
Let $N$ be a finite dimensional polynomial $B_n$-module. Then we can consider
the formal character $\mathrm{ch}_N$ of $N$ in $\Z[x_1,\dots, x_n]$. In fact
$\mathrm{ch}$ is a ring homomorphism from $\mathrm{Gr}$ to $\Z[x_1,\dots,x_n]$.
In~\cite{demazure} Demazure showed how the characters of certain
$B_n$-modules can be calculated  by applying
what is now called the \emph{Demazure operators}, $\pi_i$, $1\le i\le n-1$, to a monomial $x_1^{k_1} \dots x_n^{k_n}$.
It turns out that the operators $\pi_i$, $1\le i\le n$, define an action of the
Hecke monoid, $\rs$, on $\Z[x_1,\dots,x_n]$. Later Magyar in~\cite{magyar}
generalised Demazure's character formula for the  class of flag Weyl modules
corresponding to  percentage-avoiding shapes.
While researching this class we were lead to the idea that a  categorification of the action of $\rs$ on $\Z[x_1,\dots,x_n]$ can be useful to prove some conjectures stated in~\cite{reiner}.

In this article we show that the Hecke monoid acts
on the category of rational modules for the quantum negative Borel subgroup of the
quantum general linear group. In fact, we construct what we call a
\emph{preaction} of $\rs$ on this category. In~\cite{preaction} the second author proves that  the category of actions of $\rs$
on a category $\cat$
is equivalent to the category of preactions of $\rs$ on $\cat$.
Therefore, via this equivalence, from the constructed preaction one  can obtain an action of $\rs$ on the above referred  category. It is then quite simple to get a preaction, and so an action, of $\rs$ on the category  of $S_{\alpha,\beta}^-(n,r)$-modules, where $S_{\alpha,\beta}^-(n,r)$ is the quantised (negative) Borel-Schur algebra.
In a forthcoming paper we will show that the action of $\rs$ on the category of rational modules for the quantum Borel subgroup  induces an action of $\rs$ on the corresponding derived category. This action will provide a categorification of the action of $\rs$ on  $\Z[x_1,\dots,x_n]$.

The paper is organised as follows. In Section~\ref{second} we
  introduce the notion of a preaction  of the Hecke monoid $\rs$ on a
category $\cat$.
Section~\ref{third} contains some results,  on cotensor product and induction for
coalgebras, that are due to Takeuchi~\cite{takeuchi2} and Donkin~\cite{donkin0,donkin1}.
In Section~\ref{fourth} we study some subgroups  of  the quantum general linear group (or rather
their coordinate Hopf algebras), namely  quantum parabolic subgroups and quantum Borel
subgroups, following \cite{takeuchi} and \cite{parshall}.
We also prove the exactness of the short exact sequences~\eqref{sequence}
and~\eqref{eq:seq2}, which play a crucial role in our construction.
Section~\ref{fifth} is dedicated to the definition of functors $F_i$, $1\le i\le n-1$, that generate
an action of $\rs$ on the category of rational modules for the quantum negative
Borel subgroup.
We also construct natural isomorphisms $\tau_{ij}$, $1\le i\le j\le n-1$,  which, together with the functors
$F_i$, give a preaction of $\rs$ on the above mentioned category.  The proof that $(F,\tau)$ is in fact a preaction  is given in Section~\ref{sixth}. In Section~\ref{seventh} we show that
$(F,\tau)$ induces a preaction  of $\rs$ on the category of  $S^-_{\alpha,\beta}(n,r)$-modules.
In Section~\ref{sec:examples} we
consider explicit examples of the application of the functors $F_i$ to
$B_n$-modules.

\section{(Pre)actions of  Hecke monoids}
\label{second}

Let $n$ be a positive integer and $\Sigma_n$ the symmetric group of degree $n$.

The \emph{Hecke monoid}, $\rs$, is the
monoid   with  elements $\,T_w\,, w \in \Sigma_n\,,$ and multiplication
 determined by the rule
\begin{equation*}
T_{\sigma} T_w =
\begin{cases}
T_{\sigma w},& \mbox{if $l(\sigma w)= l(w) +1$}\\
T_w, & \mbox{if $l(\sigma w) = l(w)-1$},
\end{cases}
\end{equation*}
where $w\in \Sigma_n$ and $\sigma$ is an elementary transposition of the form
$(k,k+1)$, $1\le k\le n-1$.

Let $\cat$ be a category and $M$ a monoid with  neutral element $e$. A \emph{(pseudo)action}
$(\fcal,\lambda)$ of $M$ on $\cat$ is
\begin{enumerate}[i)]
\item a collection $\fcal$ of endofunctors $F_a\colon \cat \to \cat$, $a\in M$,
such that $F_e$ is the identity functor;
\item natural isomorphisms $\lambda_{a,b}\colon F_{a}F_b\to F_{ab}$,
such that for $a$, $b$, $c\in M$ the diagram
\begin{equation*}
\xymatrix{ F_aF_bF_c \ar[r]^-{\lambda_{a,b}F_c} \ar[d]_-{F_a \lambda_{b,c}} &
F_{ab}F_c\ar[d]^-{\lambda_{ab,c}} \\
F_a F_{bc} \ar[r]^-{\lambda_{a,bc}} & F_{abc}}
\end{equation*}
commutes, and $\lambda_{e,a} = \lambda_{a,e}$ is the identity isomorphism of
$F_a$ (see \cite{deligne}).
\end{enumerate}

Suppose $(\fcal,\lambda)$ is an action of $\rs$ on $\cat$. To simplify
notation we write $F_a$ for $F_{T_a}$ and $\lambda_{a,b}$ for
$\lambda_{T_a, T_b}$.
We also replace $(i,i+1)$ by  $i$ in the subscript of $F$ and $\lambda$.

We  define natural
isomorphisms $\tau_{ij}$, $1\le i\le j\le n-1\,,$ as follows:
\begin{equation*}
\tau_{ii} = \lambda_{i, i} \colon F_{i}^2 \to F_{i}\,;
\end{equation*}
 \begin{equation*} \tau_{ij}=\lambda_{i,j}^{-1} \lambda_{j,i}\colon F_j F_i \to F_iF_j\,, \,\,\,\mbox{ for}\,\,\, i+2\leq j\,;\end{equation*}
\begin{equation*}
\tau_{i,i+1} \colon F_{i+1}F_i F_{i+1} \to F_{i} F_{i+1}F_{i}
\end{equation*}
is the composition of
\begin{equation*}
\xymatrix@C2.2cm{ F_{i+1}F_i F_{i+1} \ar[r]^-{F_{i+1}\lambda_{i,i+1}} & F_{i+1}
F_{(i,i+1,i+2)} \ar[r]^-{\lambda_{i+1,(i,i+1,i+2)}} & F_{(i,i+2)}}
\end{equation*}
followed by the inverse of the map
\begin{equation*}
\xymatrix@C2cm{ F_{i}F_{i+1} F_{i} \ar[r]^-{F_{i}\lambda_{i+1,i}} & F_{i}
F_{(i,i+2,i+1)} \ar[r]^-{\lambda_{i,(i,i+2,i+1)}} & F_{(i,i+2)}}.
\end{equation*}

The natural transformations $\tau_{ij}$ fit in the following commutative
diagrams:
\begin{equation*}
\xymatrix{F_i^3 \ar[r]^{\tau_{ii} F_i} \ar[d]_{F_i \tau_{ii}} & F_i^2
\ar[d]^{\tau_{ii}} \\
F_i^2 \ar[r]^{\tau_{ii}} & F_i}
\end{equation*}

\vspace{2ex}

\begin{equation*}
\xymatrix@C5em{
F_{i+1}^2 F_{i}F_{i+1} \ar[rr]^-{\tau_{i+1,i+1} F_{i} F_{i+1}}
\ar[d]_-{F_{i+1} \tau_{i,i+1}}
&& F_{i+1} F_{i} F_{i+1} \ar[d]^-{\tau_{i,i+1}}\\
F_{i+1} F_i F_{i+1} F_i \ar[r]^-{\tau_{i,i+1} F_i}
& F_i F_{i+1}F_i^2 \ar[r]^-{F_iF_{i+1} \tau_{ii}} & F_i F_{i+1} F_i
 }
\end{equation*}

\vspace{6ex}

\begin{equation*}
\xymatrix@C5em{
F_{i+1} F_{i}F_{i+1}^2 \ar[rr]^-{F_{i+1} F_i\tau_{i+1,i+1} }
\ar[d]_-{ \tau_{i,i+1}F_{i+1}}
&& F_{i+1} F_i F_{i+1} \ar[d]^-{\tau_{i,i+1}}\\
F_i F_{i+1} F_i F_{i+1}  \ar[r]^-{F_i \tau_{i,i+1} }
& F_i^2 F_{i+1}F_i \ar[r]^-{F_iF_{i+1} \tau_{ii}} & F_i F_{i+1} F_i
 }
\end{equation*}

\vspace{6ex}

\begin{equation*}
\xymatrix{ F_j F_i^2 \ar[rr]^-{F_j \tau_{ii}} \ar[d]_-{\tau_{i,j} F_i} &&
F_j F_i \ar[d]_-{\tau_{ij}} \\
F_i F_j F_i \ar[r]^{F_i \tau_{i,j}} & F_i^2 F_j \ar[r]^-{\tau_{ii} F_j} & F_i
F_j}\quad\quad\quad
\xymatrix{ F_j^2 F_i \ar[rr]^-{\tau_{jj} F_i } \ar[d]^-{F_j \tau_{i,j} } &&
F_j F_i \ar[d]^-{\tau_{ij}} \\
F_j F_i F_j \ar[r]^{ \tau_{i,j}F_j} & F_i F_j^2 \ar[r]^-{F_i\tau_{jj} } & F_i
F_j}
\end{equation*}

\begin{equation*}
\xymatrix@C4em{
F_{i+1}F_i F_{i+1} F_i F_{i+1} \ar[rrr]^-{F_{i+1}F_i \tau_{i,i+1}}
\ar[ddd]_-{ \tau_{i,i+1} F_i F_{i+1}} &&&
F_{i+1} F_i^2 F_{i+1} F_i \ar[d]^-{F_{i+1} \tau_{ii} F_{i+1} F_i} \\
&&& F_{i+1} F_i F_{i+1} F_i \ar[d]^-{\tau_{i,i+1} F_i} \\
&&& F_{i} F_{i+1} F_i^2 \ar[d]^-{F_i F_{i+1} \tau_{ii}} \\
F_i F_{i+1} F_i^2 F_{i+1} \ar[r]^-{F_i F_{i+1} \tau_{ii} F_{i+1}} &
F_i F_{i+1} F_i F_{i+1} \ar[r]^-{F_i \tau_{i,i+1}} &
F_i^2 F_{i+1} F_i \ar[r]^{\tau_{ii} F_{i+1}F_i} & F_iF_{i+1} F_i }
\end{equation*}

\vspace{7ex}

\begin{equation*}
\xymatrix@C4.3em{
F_j F_{i}F_{i-1} F_{i} \ar[r]^-{\tau_{ij}F_{i-1} F_{i}} \ar[d]_-{F_j
\tau_{{i-1},i}} & F_{i} F_j F_{i-1} F_{i} \ar[r]^-{F_{i} \tau_{{i-1},j} F_{i}} &
F_{i} F_{i-1} F_j F_{i} \ar[d]_-{F_{i}F_{i-1} \tau_{ij}}  \\
F_j F_{i-1} F_{i} F_{i-1} \ar[d]_-{\tau_{{i-1},j} F_{i}F_{i-1}} &  & F_{i} F_{i-1} F_{i} F_j
\ar[d]^-{\tau_{{i-1},i}F_j} \\ F_{i-1} F_j F_{i}F_{i-1} \ar[r]^-{F_{i-1} \tau_{ij} F_{i-1}} &
F_{i-1} F_{i} F_j F_{i-1}\ar[r]^-{F_{i-1}F_{i} \tau_{{i-1},j}} & F_{{i-1}}F_{i}F_{i-1} F_j
}
\end{equation*}

\vspace{7ex}

\begin{equation*}
\xymatrix@C4.3em{
F_{{j+1}}F_{j} F_{{j+1}} F_i \ar[r]^-{F_{{j+1}}F_{j} \tau_{i,{j+1}}}
\ar[d]_-{\tau_{j,{j+1}} F_i} &
F_{{j+1}}F_{j} F_{i} F_{j+1}
\ar[r]^-{F_{{j+1}} \tau_{ij}F_{j+1}}
&F_{{j+1}}F_iF_{j}  F_{j+1}
\ar[d]^-{ \tau_{i,{j+1}} F_{j}F_{j+1}}
\\
F_{j} F_{{j+1}}F_{j} F_i
\ar[d]^-{F_{j}F_{{j+1}} \tau_{ij}}
& &
 F_{i}F_{j+1}F_{j}  F_{j+1}
\ar[d]^-{F_i \tau_{j,{j+1}}}\\
F_{j} F_{{j+1}}F_i F_{j}
\ar[r]^-{F_{j} \tau_{i,{j+1}}F_{j}}
&
F_{j}F_i F_{{j+1}} F_{j}
\ar[r]^-{ \tau_{ij} F_{j+1} F_{j}}
 &
F_i F_{j} F_{{j+1}} F_{j}
}
\end{equation*}

\vspace{7ex}

\begin{equation*}
\xymatrix {
F_k F_j F_i \ar[rr]^-{F_k \tau_{ij}}\ar[dd]_-{\tau_{jk} F_i} &&
F_k F_i F_j \ar[d]^-{\tau_{ik} F_j}\\
&& F_i F_k F_j \ar[d]^-{F_i \tau_{jk}} \\
F_j F_k F_i \ar[r]^-{F_j \tau_{ik}} & F_j F_i F_k \ar[r]^-{\tau_{ij} F_k} & F_i
F_j F_k
}
\end{equation*}

\vspace{7ex}

\begin{equation*}
\xymatrix@C10em{
F_{i+2}F_{i+1}F_{i}F_{i+2}F_{i+1}F_{i+2} \ar[r]^-{F_{i+2}F_{i+1}\tau_{i,i+2}^{-1}F_{i+1}F_{i+2}} \ar[d]_-{F_{i+2}F_{i+1}F_{i}\tau_{i+1,i+2}} &
F_{i+2}F_{i+1}F_{i+2}F_{i}F_{i+1}F_{i+2}
\ar[d]^-{\tau_{i+1,i+2}F_{i}F_{i+1}F_{i+2}} \\
F_{i+2}F_{i+1}F_{i}F_{i+1}F_{i+2}F_{i+1} \ar[d]_-{F_{i+2}\tau_{i,i+1}F_{i+2}F_{i+1}}  & F_{i+1}F_{i+2}F_{i+1}F_{i}F_{i+1}F_{i+2}  \ar[d]^-{F_{i+1}F_{i+2}\tau_{i,i+1}F_{i+2}} \\
F_{i+2}F_{i}F_{i+1}F_{i}F_{i+2}F_{i+1} \ar[d]_-{\tau_{i,i+2}F_{i+1}F_{i}F_{i+2}F_{i+1}} & F_{i+1}F_{i+2}F_{i}F_{i+1}F_{i}F_{i+2} \ar[d]^-{F_{i+1}\tau_{i,i+2}F_{i+1}F_{i}F_{i+2}} \\
F_{i}F_{i+2}F_{i+1}F_{i}F_{i+2}F_{i+1} \ar[d]_-{F_{i}F_{i+2}F_{i+1}\tau_{i,i+2}^{-1}F_{i+1}} & F_{i+1}F_{i}F_{i+2}F_{i+1}F_{i}F_{i+2} \ar[d]^-{F_{i+1}F_{i}F_{i+2}F_{i+1}\tau_{i,i+2}^{-1}} \\
F_{i}F_{i+2}F_{i+1}F_{i+2}F_{i}F_{i+1} \ar[d]_-{F_{i}\tau_{i+1,i+2}F_{i}F_{i+1}} & F_{i+1}F_{i}F_{i+2}F_{i+1}F_{i+2}F_{i} \ar[d]^-{F_{i+1}F_{i}\tau_{i+1,i+2}F_{i}}\\
F_{i}F_{i+1}F_{i+2}F_{i+1}F_{i}F_{i+1} \ar[d]_-{F_{i}F_{i+1}F_{i+2}\tau_{i,i+1}} & F_{i+1}F_{i}F_{i+1}F_{i+2}F_{i+1}F_{i} \ar[d]^-{\tau_{i,i+1}F_{i+2}F_{i+1}F_{i}}  \\
F_{i}F_{i+1}F_{i+2}F_{i}F_{i+1}F_{i} \ar[r]^-{F_{i}F_{i+1}\tau_{i,i+2}F_{i+1}F_{i}} & F_{i}F_{i+1}F_{i}F_{i+2}F_{i+1}F_{i}
}
\end{equation*}
We will say that a collection of functors $F_1$, \dots, $F_{n-1}$ and natural
isomorphisms $\tau_{ij}$, $1\le i\le j \le n-1$, satisfying the above commutative
diagrams defines a \emph{preaction} of $\rs$ on $\cat$. There is proved in
\cite{preaction}  that the category of actions of $\rs$
on $\cat$
is equivalent to the category of preactions of $\rs$ on $\cat$.

In the next sections we will construct a preaction of $\rs$ on the category of
rational modules for the  negative quantum Borel subgroups of the quantum general linear groups. Therefore, we obtain an action of $\rs$ on this category, via the above referred  equivalence.

\section{Cotensor product and induction}\label{cotensorprod}
\label{third}
In this section we collect some general definitions and results concerning
coalgebras, bialgebras, and Hopf algebras.
In our treatment of the cotensor product we follow~\cite{takeuchi2}.

We start with some notation. We  will denote by $\k$ the ground field. By a coalgebra we will always mean a $\k$-coalgebra and we
 use Sweedler summation notation for coalgebras and for comodules.

Let $C$ be a $\k$-coalgebra. By a $C$-comodule we mean a right $C$-comodule and  Comod-$C$ denotes the category of $C$-comodules.
 If $M \in$ Comod-$C$
we write $\rho_M$ for the structure map $M\to M\otimes
C$ of $M$. If $N$ is a left $C$-comodule, we denote the structure map
$N\to C\otimes N$ by $\lambda_N$.

The cotensor product $M\otimes^C N$ of these right and left comodules is defined as
 the kernel of the map
\begin{equation*}
\rho_M\otimes N - M\otimes \lambda_N \colon M\otimes N \to M\otimes C \otimes N.
\end{equation*}

If $M$ is a $C'$-$C$-bicomodule and $N$ is a $C$-$C''$-bicomodule, then
$M\otimes^C N$ is a $C'$-$C''$-bicomodule with coactions given by restricting
$\lambda_M \otimes N$ and $M \otimes \rho_N$
to $M\otimes^C N$.
\begin{remark}
\label{associativity}
If $L$ is a $C''$-$C'''$-bicomodule, then  $(M\otimes^C N) \otimes^{C''} L$ and
$M \otimes^C (N\otimes^{C''} L)$ are isomorphic, and this isomorphism is given by
restricting the natural isomorphism $(M\otimes N) \otimes L \to M\otimes
(N\otimes L)$, $(m\otimes n) \otimes l\mapsto m\otimes (n\otimes l)$.
Moreover, both $(M\otimes^C N) \otimes^{C''} L$ and $M\otimes^C
(N\otimes^{C''} L)$ can be identified with the intersection of the kernels of
the maps
\begin{equation*}
\begin{aligned}
\rho_M \otimes N \otimes L - M \otimes \lambda_N \otimes L \colon M\otimes N
\otimes L \to M \otimes C \otimes N \otimes L \\
M \otimes \rho_N \otimes L - M \otimes N \otimes \lambda_L \colon M\otimes N
\otimes L \to M \otimes N \otimes C'' \otimes L.
\end{aligned}
\end{equation*}
\end{remark}
Suppose $f\colon C\to B$ is a homomorphism of coalgebras. Then we can consider
every left (right) $C$-comodule as a
left (right) $B$-comodule via $f$. We will denote the resulting left (right)
$B$-comodule by $f_\circ M$, or simply by $M$ if no confusion arises.   In particular, we can consider $C$ as a
$B$-$C$-bicomodule.
Thus for every right $B$-comodule $M$, we get that $M\otimes^B C$ is a
$C$-comodule. The $C$-comodule $M \otimes^B C$ is called the induced comodule and
will be denoted either by $f^\circ M$ or~$\Ind_B^C M$.

If $H$ is a Hopf algebra, then the category Comod-$H$ of (right) comodules over
$H$ is endowed with a monoidal structure. Namely, if $M$ and $N$ are
$H$-comodules,  the coaction of $H$ on $M\otimes N$ is defined by
\begin{equation*}
 m\otimes n \mapsto \sum m_{(0)} \otimes n_{(0)} \otimes m_{(1)} n_{(1)}.
\end{equation*}
The trivial $H$-comodule $\ktr$ is the one-dimensional comodule with underlying vector space $\k$ and coaction given by
$1\mapsto 1\otimes 1$. It is clear that $\ktr$ can be chosen as the identity
object for the above tensor product.

We will frequently use the tensor identity:

\begin{theorem}[{\cite[Proposition~1.3]{donkin1}}]\label{tensor}
Suppose $f\colon H_1\to H_2$ is a homomorphism of Hopf algebras and $M$ is an
$H_1$-comodule. Then for every $H_2$-comodule $N$ there are natural isomorphisms
\begin{equation*}
R^k \Ind_{H_2}^{H_1} ( M\otimes N) \cong M \otimes R^k\Ind_{H_2}^{H_1} N,\quad
\mbox{for all } k\ge 0.
\end{equation*}
\end{theorem}
The explicit formula for the above isomorphism when $k=0$ is given in
\cite{donkin0}.
Namely, we have
\begin{equation}\label{Lambda1}
\begin{aligned}
\phi\colon M\otimes \Ind_{H_2}^{H_1} N & \to \Ind_{H_2}^{H_1}(M\otimes N)\\
x \otimes \sum_i y_i \otimes c_i &\mapsto \sum_i x_{(0)} \otimes y_i \otimes
x_{(1)} c_i.
\end{aligned}
\end{equation}

\section{Quantisation}
\label{fourth}
In this section we study  properties of some quantised bialgebras and of subgroups of  quantum general linear groups.
We start with an overview of notions introduced in \cite{takeuchi} and some
results proved in \cite{parshall}.

From here on, $\alpha$ and $\beta$  are non-zero elements of
$\k$.
We will also denote by $\n$ the set of integers~$\{1,\cdots ,n\}$.

Let $F(n)$ be the free $\k$-algebra with $ n^2$ generators
$x_{ij}$,  for $ i,j \in \n$.  Denote by $I_{\alpha, \beta}$ the ideal of $F(n)$
generated by elements of the form
\begin{equation}
\label{relations}
\begin{aligned}
x_{is}x_{ir}&-\alpha x_{ir}x_{is},&&  \hbox{ for  } 1\leq i\leq n
\mbox{ and } 1\leq  r<s\leq n;\\
x_{jr}x_{ir}&-\beta x_{ir}x_{jr}, &&\hbox{ for  } 1\leq i<j \leq n
\mbox{ and } 1\leq  r\leq n;\\
x_{jr}x_{is}&- \alpha^{-1} \beta x_{is}x_{jr},&& \hbox{ for } 1\leq i <j \leq n \mbox{ and } 1\leq
r< s < n;\cr
x_{js}x_{ir}&-x_{ir}x_{js}-(\beta -\alpha^{-1})x_{is}x_{jr},&& \hbox{ for }  1\leq i<j \leq n
\mbox{ and } 1\leq r<s\leq n.
\end{aligned}
\end{equation}
The algebra
$\left.\raisebox{0.3ex}{$F(n)$}\middle/\!\raisebox{-0.3ex}{$I_{\alpha,\beta}$}\right.$
is denoted by $A_{\alpha,\beta}(n)$ and the canonical image $x_{ij} +
I_{\alpha,\beta}$ of
$x_{ij}$
 in $A_{\alpha,\beta}(n)$ by $c_{ij}$.
In what follows, we will often skip the subscripts $\alpha$,~$\beta$.
For a matrix $\omega \in M_n(\N)$
we write $c^\omega$ for the product
\begin{equation*}
c_{11}^{\omega_{11}} c_{12}^{\omega_{12}} \dots c_{1n}^{\omega_{1n}} \dots
c_{nn}^{\omega_{nn}}
\end{equation*}
and, similarly, $x^\omega$ for the product
\begin{equation*}
x_{11}^{\omega_{11}} x_{12}^{\omega_{12}} \dots x_{1n}^{\omega_{1n}} \dots
x_{nn}^{\omega_{nn}}.
\end{equation*}
 On the set $\left\{\, x_{is} \,\middle|\, i,s \in \n
\right\} \subset F(n)$ we define an ordering  by
$x_{js}>x_{it}$ if $j>i$, and $x_{is}>x_{ir}$ if $s>r$.  We consider the corresponding lexicographical ordering on the set of monomials~$\left\{\, x^{\omega} \,\middle|\, \omega\in M_n(\N)
\right\}$.

The following fact is well-known but  we include a short sketch of a proof.
\begin{theorem}\label{groebner}
 The set \eqref{relations} is a Gr\"obner basis of $I_{\alpha,\beta}$ with respect to the above
ordering. Moreover, $\left\{\, c^\omega \,\middle|\, \omega\in M_n(\N) \right\}$
is a basis of $A(n)$.
\end{theorem}
\begin{proof}
Note that every element in \eqref{relations} is written so that
the leading monomial is a first term. To show that \eqref{relations} is a Gr\"obner
basis one has to check that all the critical pairs are resolvable.

Suppose $m\ge 3$.
Let  $S'$ be the set \eqref{relations} for $n=3$, and $S$  the
set~\eqref{relations} for $n=m$.
It is easy to
see that every critical pair of $S$ involves at most three row and three column
indices.
 Let us fix  two triples of indices $1\le i_1<i_2<i_3\le m$ and $1\le r_1<r_2<r_3\le m$. Then we
have a homomorphism of free algebras  $\varphi\colon F(3)\to F(m)$ defined by
$\varphi(x_{js}) = x_{i_j,r_s}$.
Clearly $\varphi(S') \subset S$.  Now every critical pair involving row indices
$i_1$, $i_2$, $i_3$, and column indices $r_1$, $r_2$, $r_3$ lies in
$\phi(S')$.  This shows that  it is enough to prove that $S'$
is a Gr\"obner basis.
Hence the claim of the theorem has to be verified for  $n\le 3$ only.
The case $n=1$ is
trivial, and the cases $n=2$, $n=3$  can be
checked using a computer algebra system, such
as Magma~\cite{magma}.

It is easy to see that  the set $\left\{\, x^\omega \,\middle|\,
\omega\in M_n(\N)
\right\}$  is the set of non-reducible monomials with respect to the Gr\"obner basis
\eqref{relations}. Thus $\left\{\, c^{\omega} \,\middle|\, \omega\in
M_n(\N) \right\}$ is a basis of $A(n)$.
\end{proof}

Given a sequence   $b = (b_1,\dots,b_n)\in \n^n$, let $\,I(b)=I_{\alpha,\beta}(b)\,$ be the ideal of $A(n)$ generated by $\,\left\{ c_{is} \middle| s >
b_i
\right\}.$
 We define  the
quotient algebra
$$A(b)=\left.\raisebox{0.3ex}{$A(n)$}\middle/\!\raisebox{-0.3ex}{$I(b)$}\right.,$$
\noindent and denote by $[c_{is}]_b$ the image of $c_{is}$
under the canonical projection  from $A(n)$ to $A(b)$.

\begin{theorem}
\label{basis}
 On the set $\left\{\, x^{\omega} \,\middle|\, \omega
\in M_n(\N)
\right\}\subset F(n)$  consider the  ordering used
in Theorem~\ref{groebner}. Let $b = (b_1,\dots,b_n) \in \n^n\,$ be a non-decreasing sequence. Denote by $S''$ the union of the set \eqref{relations} and
$S':=\left\{\, x_{is} \,\middle|\, s>b_i \right\}$. Then $S''$ is a Gr\"obner
basis of the ideal generated by $S''$. In particular,
 $$\left\{\, [c^\omega]_b \,\middle|\,  \omega\in M_n(\N),\, \omega_{is} =0 \,\,\mbox{ for}\,\, s > b_i
\right\}$$  is a basis of $A(b)$.
\end{theorem}
\begin{proof}
To prove the theorem we have to check that all ambiguities in $S''$
are resolvable. Let us write $S$ for the set \eqref{relations}. For pairs of
elements in $S$ the ambiguities are resolvable, since $S$ is a Gr\"obner basis, by
Theorem~\ref{groebner}. There are no ambiguities between pairs of
elements in $S'$. Thus we only have
 to check that all the ambiguities between an element in $S$ and an element
in $S'$ are resolvable. The only
interesting case is when the element of $S$ is of the form $x_{js}x_{ir} -
x_{ir} x_{js} - (\beta-\alpha^{-1}) x_{is} x_{jr}$, for $i<j$ and $r<s$,  and the element of
$S'$ is either $x_{js}$ or $x_{ir}$. In the first
case we get that $s>b_j\ge b_i$ and therefore also $x_{is} \in S'$.
In the second case $s>r>b_i$ and again $x_{is} \in S'$. Therefore in both
cases the ambiguity is resolvable.

It is now  straightforward that $\left\{\, x^\omega \,\middle|\, \omega\in M_n(\N),\, \omega_{is} =0 \,\,\mbox{ for}\,\, s > b_i  \right\}$ is the set of non-reducible monomials with respect to $S''$. Thus   the set
$\left\{\, [c^{\omega}]_b \,\middle|\,  \omega\in M_n(\N),\, \omega_{is} =0 \,\,\mbox{ for}\,\, s > b_i \right\}$ is a $\k$-basis
of $A(b)$.
\end{proof}
\begin{corollary}
\label{intdomain}
Suppose $b=(b_1,\dots,b_n)\in \n^n$ is a non-decreasing sequence.
Then the algebra $A(b)$ has no zero divisors.
\end{corollary}
\begin{proof}
Consider the subset $S''$ of $ F(n)$ defined in Theorem~\ref{basis}. We have
that $S''$ is a Gr\"obner basis of the
kernel of the canonical projection $ F(n)\twoheadrightarrow A(b)$. One can check
that the leading term of (the reduced expression for) $[c^{\omega}]_b[ c^{\tau}]_b$ is
$[c^{\omega + \tau}]_b$ multiplied by $\alpha^s \beta^t$, for suitable
$s$, $t\in \N$.
Thus given two non-zero elements in $A(b)$ with leading monomials $[c^\omega]_b$ and
$[c^\tau]_b$, respectively, we get that  their product has leading monomial
$[c^{\omega + \tau}]_b$ and so it is non-zero.
\end{proof}

The algebra $A(n)$ has a unique
structure of  bialgebra with comultiplication $\comult\colon A(n) \to
A(n)\otimes A(n)$ and counit $\counit\colon A(n)\to \k$, satisfying
\begin{equation*}
\comult(c_{ij}) = \sum_{k=1}^n c_{ik} \otimes c_{kj}, \quad \counit(c_{ij}) =
\begin{cases}
1, & i=j\\
0, & i\not=j.
\end{cases}
\end{equation*}

The next theorem, proved in~\cite{parshall}, allows us to identify the coalgebra
$A(n)= A_{\alpha, \beta}(n)$ with the coalgebra $A_{1, \alpha \beta}(n)$
studied in  \cite{dipperdonkin, donkin1}.

For a matrix $\omega \in M_n(\N)$, we denote by $J(\omega)$
the number
\begin{equation*}
\sum_{i<j, s<t} \omega_{it} \omega_{js}.
\end{equation*}

\begin{theorem}[{\cite[Proposition~2.1]{parshall}}]\label{bialgebra-isomorphism}
Suppose that $\alpha$, $\beta$, $\alpha'$, $\beta'$ are non-zero elements in
$\k$ such that $\alpha'\beta' = \alpha\beta$.
Then the map
\begin{equation*}
\begin{aligned}
A_{\alpha,\beta}(n) & \to A_{\alpha',\beta'} (n) \\
c^{\omega} &\mapsto (\alpha/\alpha')^{J(\omega)} c^{\omega}
\end{aligned}
\end{equation*}
is an isomorphism of coalgebras.
\end{theorem}

Before we proceed, we need to introduce some notation concerning sequences of
natural numbers. We denote by  $v_{l}$ the $n$-tuple $(0, \cdots,0,1,0 , \cdots , 0)$ ($1$ in the $l$th position). Given a composition $\lambda = (\lambda_1, \dots, \lambda_m)$
of $n$, we write
\begin{equation}
\label{composition}
(\lambda_1^{\lambda_1},
(\lambda_1+\lambda_2)^{\lambda_2}, \dots, n^{\lambda_m})
\end{equation}
for
\begin{equation*}
(\lambda_1, \dots, \lambda_1, \lambda_1 + \lambda_2, \dots,
\lambda_1+\lambda_2,\dots, n,\dots,n),
\end{equation*}
where
$\lambda_1 + \dots + \lambda_k$ is repeated $\lambda_k$ times.
For $\lambda = (1^n)$, we obtain the sequence $\delta = (1,2,\dots,n)$.
If $\lambda = (1^{l-1}, 2, 1^{n-l-1})$, for some natural number $1\leq l \leq n-1$, we denote the corresponding sequence by
$a[l]$. Thus
\begin{equation*}
a[l] = (1,2,\dots, l-1,l+1,l+1,l+2,\dots, n) = \delta + v_{l}.
\end{equation*}
\begin{remark}
\label{rem:bialgebras}
Suppose that $b = (\lambda_1^{\lambda_1},
(\lambda_1+\lambda_2)^{\lambda_2}, \dots, n^{\lambda_m})
$ for some composition $\lambda$ of $n$.
By \cite[Proposition~2.3]{donkin1}.
$I_{1,q}(b)$ is a biideal of $A_{1,q}(n)$, $q\in \k^*$. Combining this fact
with Theorem~\ref{bialgebra-isomorphism}, we see that
$I_{\alpha,\beta}(b)$ is a coideal, for every $\alpha$, $\beta\in \k^*$. As
$I_{\alpha,\beta}(b)$ is an ideal by definition, we get that
$A_{\alpha,\beta}(b)$ is a bialgebra.
In particular, $A_{\alpha,\beta}[l] := A_{\alpha,\beta} (a[l])$ and
$A_{\alpha,\beta}(\delta)$ are bialgebras.
\end{remark}
\begin{proposition}
\label{first}
Suppose that $b=(b_1,\dots,b_n)\in \n^n$ satisfies $b_k\not= l$ for all $k$. Then
$\comult(I(b)) \subset I(b)\otimes A(n) + A(n)\otimes I(\al)$. In particular,
$A(b)$ is an $A[l]$-comodule with the coaction given by
\begin{equation*}
\begin{aligned}
A(b) & \to A(b)\otimes A[l]\\
[x]_b & \mapsto\sum [x_{(1)}]_b \otimes [x_{(2)}]_{a[l]}.
\end{aligned}
\end{equation*}
\end{proposition}
\begin{proof}
Let $c_{ij}\in I(b)$. Then $j>b_i$ and
\begin{equation*}
\comult(c_{ij}) = \sum_{k>b_i} c_{ik}\otimes c_{kj} + \sum_{k=1}^{b_i}
c_{ik}\otimes c_{kj}.
\end{equation*}
For $k>b_i$, we get that $c_{ik} \in I(b)$, and so the elements of the first sum
are in $I(b)\otimes A(n)$. For $k\le b_i$, we get $j>b_i\ge k$.
It is easy to see that $c_{kj}\in I(\al)$ if and only if $j>k$ and
$(k,j)\not=(l,l+1)$. Suppose $(k,j)=(l,l+1)$. Then
\begin{equation*}
j = l+1 > b_i \ge k = l
\end{equation*}
implies $b_i = l$, which contradicts our assumption on $b$. Therefore
$c_{kj}\in I(\al)$ for all $k\le b_i$.
\end{proof}
With a proof similar to the above one, we obtain the following result for $A(b)$ and the bialgebra~$A(\delta)$.
\begin{proposition}
For any  $b= (b_1,\dots, b_n)\in \n^n$, we have $\comult(I(b)) \subset I(b)\otimes
A(n) + A(n)\otimes I(\delta)$. In particular, $A(b)$ has a structure of
$A(\delta)$-comodule
with the coaction given by
\begin{equation*}
[x]_{b} \mapsto \sum [x_{(1)}]_{b}\otimes [x_{(2)}]_{\delta}.
\end{equation*}
\end{proposition}

Note that the $A(\delta)$-coaction on $A(b)$ is multiplicative in the sense of
the following proposition.
\begin{proposition}
\label{measure}
 Denote by $\coaction$ the coaction of $A(\delta)$ both on $A(b)$ and on  $A(b)\otimes
A(b)$, and by $\mu$ the multiplication in the algebras $A(n)$ and $A(b)$. Then the following diagram is commutative
\begin{equation*}
\xymatrix{
A(b) \otimes A(b) \ar[r]^-{\coaction} \ar[d]_{\mu} &
 A(b)\otimes A(b) \otimes A(\delta) \ar[d]^{\mu\otimes A(\delta)}\\
A(b) \ar[r]^-{\coaction} & A(b) \otimes A(\delta).
}
\end{equation*}
\end{proposition}
\begin{proof}
Let  $h$ denote the following composition of maps
\begin{equation*}
\xymatrix@R6ex@C6em{
A(n) \otimes A(n) \ar[r]^-{\comult\otimes \comult} &
A(n) \otimes A(n) \otimes A(n) \otimes A(n)
\ar`[l]+`[l]+`[ld]_-{\tau_{23}}[ld] &\\
A(n) \otimes A(n) \otimes A(n) \otimes A(n) \ar[r]^-{A(n)\otimes A(n) \otimes \mu} &
A(n) \otimes A(n) \otimes A(n) ,
}
\end{equation*} where $\tau_{23}$ is the twist map of the second and third
factors of $A(n) \otimes A(n) \otimes A(n) \otimes A(n)$.
Then we have the  diagram
\begin{equation*}
\xymatrix@C1em{
A(b)\otimes A(b) \ar[rrr]^-{\coaction } \ar[ddd]_{\mu} & & &
A(b)  \otimes A(b)\otimes A(\delta) \ar[ddd]^{\mu\otimes A(\delta)}
\\ &
A(n) \otimes A(n) \ar@{->>}[lu] \ar[r]^-h  \ar[d]_{\mu} &  A(n) \otimes A(n) \otimes
A(n) \ar@{->>}[ru] \ar[d]^{\mu \otimes A(n)}
  \\
& A(n) \ar[r]^{\comult}\ar@{->>}[ld] & A(n) \otimes A(n)\ar@{->>}[rd]
\\ A(b) \ar[rrr]^-{\coaction}  & &  & A(b) \otimes A(\delta).
}
\end{equation*}
The internal square in the above diagram commutes since $A(n)$ is a bialgebra.
The trapezoids commute by the definition of the $A(\delta)$-coaction on $A(b)$ and of
the multiplication on $A(b)$. Since the upper-left diagonal arrow is surjective, we
conclude that the exterior square is also commutative.
\end{proof}

Let  $a= (a_1, \cdots, a_n)\in \N^n$.
We denote by   $\k_a$     the  1-dimensional  $A(\delta)$-comodule with underlying vector space $\k$  and structure map
\begin{equation}
\label{onedim}
\begin{aligned}
\k_a& \to \k_a \otimes A(\delta)\\
1 \ & \mapsto 1 \otimes [c_{11}^{a_1}\cdots c_{nn}^{a_n}]_\delta.
\end{aligned}
\end{equation}
Given  $b= (b_1,\dots, b_n) \in \n^n$, consider the map
\begin{align*}
f\colon  A(b)\otimes \k_{v_{b_l}}& \to A(b)\\
x \otimes 1 & \mapsto x[c_{l,b_l}]_b.
\end{align*}
\begin{proposition}
\label{prop:injective}
If $b$ is non-decreasing, then the  map $f$ defined above is an injective homomorphism of $A(\delta)$-comodules.
\end{proposition}
\begin{proof}
Note that, since there are no zero divisors in $A(b)$, $f $ is injective.

Denote by $\coaction$ the coaction of $A(\delta )$ on $A(b)$ and by $\rho_{b_l}$  the coaction of $A(\delta)$ on  $A(b)\otimes \k_{v_{b_l}}$.
Then, for any $x \in A(b)$, we have
\begin{equation*}
\rho_{b_l} (x\otimes 1) = \sum ([x_{(1)}]_b \otimes 1 ) \otimes [x_{(2)}]_\delta
[c_{b_l,b_l}]_\delta.
\end{equation*} Hence
\begin{equation}\label{axbl}
 f\otimes A(\delta)\left(\rho_{b_l} (x\otimes 1)\right)= \sum [x_{(1)}]_b[c_{l,b_l}]_b \otimes [x_{(2)}]_\delta
[c_{b_l,b_l}]_\delta.
\end{equation}
Further
\begin{equation*}
\label{clblb}
\coaction ([c_{l,b_{l}}]_b) = \sum_{k=1}^n [c_{lk}]_b\otimes [c_{k,b_l}]_\delta.
\end{equation*}
Now we have  $[c_{k,b_l}]_\delta =0$ for $k<b_l$ , and
$[c_{lk}]_b=0$ for $k>b_l$. Therefore  $\coaction ([c_{l,b_l}]_b) = [c_{l,b_l}]_b \otimes
[c_{b_l,b_l}]_\delta$.
Using this and Proposition~\ref{measure}, we get
\begin{equation}
\label{ackblb}
\coaction (x[c_{l,b_l}]_b) =  \sum [x_{(1)}]_b[c_{l,b_l}]_b \otimes [x_{(2)}]_\delta
[c_{b_l,b_l}]_\delta.
\end{equation}
Compairing \eqref{axbl} and \eqref{ackblb}, we see that $f$ is indeed a
homomorphism of $A(\delta)$-comodules.
\end{proof}

\begin{proposition}\label{minha}
Suppose that $b$ and $b-v_l \in \n^n$ are non-decreasing sequences. Then we have the following short
exact sequence
of $A(\delta)$-comodules
\begin{equation}\label{sequence}
\xymatrix{ 0 \ar[r] & A(b)\otimes \k_{v_{b_l}} \ar[r]^-{f} & A(b)
\ar[r]^-{\pi} & A(b-v_l)\ar[r] & 0}.
\end{equation}
\end{proposition}
\begin{proof}

Clearly $I(b) \subset I(b-v_l)$. So we can consider the canonical projection $\pi\colon A(b)\twoheadrightarrow A(b-v_l)$.

By Proposition~\ref{basis} the sets $\left\{\, [c^\omega]_b \,\middle|\,
\omega_{ij}=0,\ j>b_i
\right\}$ and
\begin{equation*}
\left\{\, [c^\omega]_{b-v_l} \,\middle|\, \omega_{ij}=0,\ j>b_i
\,\,\mbox{if}\,\,i\neq l, \,\,\mbox{and}\,\,j> b_l-1, \,\,\mbox{if}\,\,i=l
 \right\}
\end{equation*}
 are basis of $A(b)$ and $A(b-v_l)$, respectively.
Therefore $\left\{\,[c^\omega]_b\,\middle|\,
\omega_{ij} = 0,\ j>b_i;\ \omega_{l,b_l}\not=0
\right\}$ is a basis of the kernel of $\pi$.

Let $\omega\in M_n(\N)$ be such that $\omega_{ij} =0$   for $j>b_i$. Then,  in
particular, $\omega_{l,b_l}$ is the last possible non-zero element in the
$l$th row of $\omega$. Define $\omega'\in M_n (\N)$ to be the matrix with the
same elements in the first $l$ rows as $\omega$ and zeros elsewhere. Denote
$\omega - \omega'$ by $\omega''$. Then
from the definition of $c^\omega$, we get $[c^\omega]_b = [c^{\omega'}]_b
[c^{\omega''}]_b $. Moreover, $[c^{\omega'}]_b [c_{l,b_l}]_b
[c^{\omega''}]_b = [c^{\omega + e_{l,b_l}}]_b $, where $e_{l,b_l}$ denotes the matrix with $1$ in position $(l,b_l)$ and zeros elsewhere. We claim that $[c^{\omega''}]_b
[c_{l,b_l}]_b = \alpha^s\beta^t [c_{l,b_l}]_b [c^{\omega''}]_b$ for suitable
integers $s$ and
$t$. In fact $[c^{\omega''}]_b$ is the product of the elements $[c_{ij}]_b$ with $i> l$. If $j< b_l$, we get
\begin{equation*}
[c_{ij}]_b [c_{l,b_l}]_b = \alpha^{-1} \beta [c_{l,b_l}]_b [c_{ij}]_b.
\end{equation*}
If $j=b_l$, then
\begin{equation*}
[c_{ij}]_b [c_{l,b_l}]_b = [c_{i,b_l}]_b [c_{l,b_l}]_b = \beta [c_{l,b_l}]_b
[c_{ij}]_b.
\end{equation*}
If $j>b_l$, then
\begin{equation*}
[c_{ij}]_b [c_{l,b_l}]_b =  [c_{l,b_l}]_b [c_{ij}]_b + (\beta-\alpha^{-1}) [c_{lj}]_b
[c_{i,b_l}]_b.
\end{equation*}
Since $j>b_l$, we get that $[c_{lj}]_b = 0$. Thus
\begin{equation*}
[c_{ij}]_b [c_{l,b_l}]_b =  [c_{l,b_l}]_b [c_{ij}]_b
\end{equation*}
in this last case. Therefore, we have $[c^{\omega}]_b
[c_{l,b_l}]_b = \alpha^s\beta^t [c^{\omega + e_{l,b_l}}]_b\,.$
This shows that the image of  $f$  and the kernel of $\pi$ coincide.
By Proposition~\ref{prop:injective}, the map $f$ is injective and
so~\eqref{sequence} is exact.
\end{proof}
Let $l(\sigma)$ denote the length of the permutation $\sigma \in
\Sigma_n$.
The quantum determinant is the element of $A(n)$ defined by
\begin{equation*}
\begin{aligned}
d = d_{\alpha,\beta} & = \sum_{\sigma\in \Sigma_n} (-\alpha)^{-l(\sigma)}
c_{1,\sigma(1)}c_{2,\sigma(2)}\dots
c_{n,\sigma(n)}\\
& = \sum_{\sigma\in \Sigma_n} (-\beta)^{-l(\sigma)}  c_{\sigma(1),1}
c_{\sigma(2), 2} \dots c_{\sigma(n),n}.
\end{aligned}
\end{equation*}
The determinant $d$ is a group-like element of $A(n)$, see~\cite{parshall}.
For every nondecreasing  $b\in \n^n$ such that $b_i\ge i$, we get that $[d]_b$ is a non-zero element of
$A(b)$ and so a non-zero divisor, by Corollary~\ref{intdomain}.
We also have $[c_{ij}]_b [d]_b = (\alpha^{-1}\beta)^{i-j} [d]_b [c_{ij}]_b$. Hence, we can
localize $A(b)$
with respect to $d$.
We will denote the resulting localization by $A(b)_d$. 
\begin{remark}
\label{rem:first}
Since $d$
is group-like this localization process preserves the coalgebra and comodule
structures. Therefore, $A(n)_d$, $A(\delta)_d$, $A[l]_d$ are bialgebras, $A(b)_d$
is an $A(\delta)_d$-comodule, and for $b$ such that $b_i\not= l\,,$ for all
$i$, $A(b)_d$ is an $A[l]_d$-comodule.
\end{remark}
The bialgebra $A(n)_d$ admits a Hopf algebra structure with the antipode given
by
\begin{equation*}
\antipode (c_{is})  = (-\beta)^{s-i} d^{-1} d_{si},
\end{equation*}
where $d_{si}$
denotes the quantum determinant of the subalgebra of $A(n)$ obtained by deleting
all generators $c_{sk}$
 and $c_{ki}$ with $1\le k\le n$  (see~\cite[(1.7)]{parshall}).

\begin{proposition}
Let $\lambda = (\lambda_1, \dots, \lambda_m)$ be a composition of $n$, and
\begin{equation*}
 b = (\lambda_1^{\lambda_1}, (\lambda_1 +
\lambda_2)^{\lambda_2}, \dots, n^{\lambda_m}).
\end{equation*}
Then, the kernel $J(\lambda)$ of the canonical projection $A(n)_d \to A(b)_d$ is a
Hopf ideal generated, as an ideal, by $\left\{\, c_{is} \,\middle|\, s> b_i
\right\}$.
Therefore $A(b)_d$ admits a Hopf algebra structure with the antipode given by
\begin{equation*}
\antipode ([c_{is}]_b)  = (-\beta)^{s-i} [d]_b^{-1} [d_{si}]_b.
\end{equation*}
\end{proposition}
\begin{proof}
We know that, in this case, $A(b)_d$ is a bialgebra. It is obvious that the projection $A(n)_d \to A(b)_d$ is a homomorphism of
bialgebras, which implies that $J(\lambda)$ is  a biideal.

Suppose $d^{-k} y \in J(\lambda)$, for some $y \in A(n)$. Then $[d]_b^{-k} [y]_b =0$
in $A(b)_d$. By the definition of localization this implies that $[y]_b=0$, and so
$y \in I(b)$. Therefore $y = \sum_{i=1}^n \sum_{s>b_i} y_{is} c_{is}
y'_{is}$ for some elements $y_{is}$, $y'_{is} \in A(n)$. Since $d^{-k}
y_{is} \in A(n)_d$, we get that the ideal $J(\lambda)$ is generated by the elements
$c_{is}$
 with $s>b_i$.

As $S$ is an anti-endomorphism of $A(n)$, to show that $S(J(\lambda))\subset
J(\lambda)$ it  is enough to check that $S(c_{is}) \in J(\lambda)$, for every
pair $(i,s)$ such that $s>b_i$.
But as $S(c_{is}) = (-\beta)^{s-i} d^{-1} d_{si}$ it is sufficient to verify that
$d_{si} \in I(b)$.

Let us consider the embedding $\varphi\colon A(n-1) \to A(n)$ determined by
\begin{equation*}
c_{jt} \mapsto
\begin{cases}
c_{jt}, & j<s,\ t<i\\
c_{j+1,t} & j\ge s, \ t<i\\
c_{j,t+1}, & j<s,\   t\ge i\\
c_{j+1,t+1}, & j\ge s,\ t\ge i.
\end{cases}
\end{equation*}
Then, by the definition of $d_{si}$, we get that $d_{si}$ is the image of the
determinant $d\in
A(n-1)$ under $\varphi$, see~\cite{parshall}. Now
consider the ideal $\varphi^{-1}(I(b))$. Suppose
$k$ and $w\in \left\{ 1,\dots,m \right\}$ are such that
\begin{equation*}
\begin{aligned}
\lambda_1 + \dots + \lambda_{k-1} < i \le \lambda_1 + \dots + \lambda_{k}\\
\lambda_1 + \dots + \lambda_{w-1} < s \le  \lambda_1 + \dots + \lambda_{w}.
\end{aligned}
\end{equation*}
In other words $(i,s)$ lies in the $(k,w)$ block determined by the composition
$\lambda$.
Note that $s > b_i\geq i$ implies that $w >k$.
One can show that
$\varphi^{-1}(I(b)) = I(b')$ where
\begin{equation*}
b' = (\lambda_1^{\lambda_1}, \dots, (\lambda_1 + \dots + \lambda_k
-1)^{ \lambda_k}, \dots, (\lambda_1 + \dots +
\lambda_w-1)^{ \lambda_w -1}, \dots, (n-1)^{\lambda_m}).
\end{equation*}
Suppose $d\not \in I(b')$. Then there is $\sigma \in \Sigma_{n-1}$
such that $c_{j,\sigma(j)} \not \in I(b')$, for all $1\le j\le n-1$. This implies
that, for all $1\le j\le \lambda_1 $, we must have $1\le \sigma (j) \le \lambda_1$. In other words, $\sigma$ maps bijectively $[1,\lambda_1] \cap \N$ into itself.
Now, for $\lambda_1 + 1 \le j \le \lambda_1 + \lambda_2$, we must have
$1 \le  \sigma (j) \le \lambda_1 + \lambda_2$. But since $\sigma$ maps
$[1,\lambda_1] \cap \N$ maps bijectively into itself, this implies that $\sigma$
also maps $[\lambda_1 + 1, \lambda_1 + \lambda_2 ] \cap \N$ into itself.
Proceeding this way, we get that $\sigma$ must map $[\lambda_1 + \dots +
\lambda_{k-1} + 1, \lambda_1 + \dots + \lambda_k] \cap \N$
bijectively into
$[\lambda_1 + \dots +
\lambda_{k-1} + 1, \lambda_1 + \dots + \lambda_k-1] \cap \N$,
which is impossible. Therefore, we get that $d\in I(b')$ and thus $d_{si} = \varphi
(d) \in I(b)$.
\end{proof}

The Hopf algebra
$A(n)_d=(A(n)_{\alpha,\beta})_d$ is the coordinate algebra of the quantum
general linear group $\gl_{\alpha,\beta}(n,\k)$, defined by  Takeuchi
in~\cite{takeuchi} and also studied in~\cite{parshall}.  The quantum groups
$\gl_{1,\beta}(n,\k)$ and $\gl_{\beta,\beta}(n,\k)$  are, respectively, the
quantum general linear groups studied by Dipper and Donkin in~\cite{dipperdonkin} and Parshall and Wang in~\cite{parshall:book}.

If $\lambda = (\lambda_1, \dots, \lambda_m)$ is a composition of $n$, and $
 b = (\lambda_1^{\lambda_1}, (\lambda_1 +
\lambda_2)^{\lambda_2}, \dots, n^{\lambda_m})$, then the Hopf algebra
$A(b)_d$ can be considered  as the coordinate algebra of a \emph{quantum
parabolic subgroup} of  $\gl_{\alpha,\beta}(n,\k)$. Taking $b=\delta$, we obtain
the coordinate algebra $A(\delta)_d$ of the \emph{quantum negative Borel
subgroup}  of $\gl_{\alpha,\beta}(n,\k)$.

Quantum parabolic and Borel subgroups were extensively studied by Donkin in~\cite{donkin1} (see also~\cite{DSY}), for the case $\alpha =1$.

Consider  $a= (a_1, \cdots, a_n )\in \n^n$. Then
we also denote by   $\k_a$     the  1-dimensional  $A(\delta)_d$-comodule which
is the restriction of the $A(\delta)$-comodule  $\k_a$ defined
in~\eqref{onedim}.

Given $a=(a_1, \cdots, a_n),\,\, b=(b_1, \cdots, b_n) \in \n^n$, we write $b\geq a$ if $b_i\geq a_i$, for $ i \in \n.$
Then we have the following extension of Proposition~\ref{minha}

\begin{proposition}
\label{seq2}
Suppose $b\in \n^n$ is such that $b\ge a[l]$, and $b$, $b-v_l$ are non-decreasing  sequences. Then we
have an exact sequence of $A(\delta)_d$-comodules
\begin{equation}
\label{eq:seq2}
\xymatrix{ 0 \ar[r] & A(b)_d \otimes \k_{v_{b_l}} \ar[r]^-f & A(b)_d \ar[r]^-\pi &
A(b-v_l)_d \ar[r] & 0 },
\end{equation}
where $f$ is the comodule homomorphism defined
by $z\otimes 1  \mapsto z[c_{l,b_l}]_b$, for any $z \in A(b)_d$,  and $\pi $ is the canonical projection.
\end{proposition}
\begin{proof}
It is obvious that $\pi$ is surjective. Let $[d]_b^{-k} x \in \ker(\pi)$, with
$x\in A(b)$ and $k\in \N$. Then
$[d]_{b-v_l}^{-k}\pi(x) = 0$
in $A(b-v_l)_d$. Since $[d]_{b-v_l}$
is not a zero divisor in $A(b-v_l)$, we get that $\pi(x)=0$ in $A(b-v_l)$. This
shows that $x$ is in the kernel of the projection $A(b) \to A(b-v_l)$. Since
\eqref{sequence} is exact, we get that there is $y\otimes 1\in A(b)\otimes \k_{v_l}$ such
that $f(y\otimes 1) =x$. Therefore $f([d]_b^{-k} y\otimes 1) =
[d]_b^{-k}x$. This shows that \eqref{eq:seq2} is
exact at the second term.

Now, suppose $[d]_b^{-k}y\otimes 1\in \ker(f)$, with $y\in A(b)$. Then
$[d]_b^{-k} y c_{l,v_l} =0$ in $A(b)_d$. Thus $y c_{l,v_l} =0$
in $A(b)$. Since $c_{l,v_l}$ is not a zero-divisor in $A(b)$, we get $y =0$ in
$A(b)$. Therefore $f$ is injective.
\end{proof}

\section{The construction of the preaction}\label{section4}
\label{fifth}

Our next step will be to define a preaction of $\rs$ on the category Comod-$A(\delta)_d$.

For any $1 \leq i\leq n-1$, let $\pi_i \colon A[i]_d \to A(\delta)_d$ be  the canonical projection. We denote
  the corresponding induction functor  $\Ind_{A(\delta)_d}^{A[i]_d}$ by $\ind{i}$. For any $A[i]_d$-comodule $M$, we also write $M$ for the restricted $A(\delta)_d$-comodule $ \pi_{i_{\circ}} (M).$

Define $F_i$ as the functor $\ind{i}$
followed by the restriction to $A(\delta)_d$, i.e,
$$F_i= \pi_{i_{\circ}} \ind{i}\colon \,\,\, \mbox{Comod-}A(\delta)_d\,\to \,\,\, \mbox{Comod-}A(\delta)_d.$$ Thus every $F_i$ is an endofunctor of
Comod-$A(\delta)_d$. Next we will define natural isomorphisms
$\tau_{ij}$,  $1\le i\le j\le n-1$, and, in  Section~\ref{section6}, we will prove that they satisfy all the
necessary commutation relations to define a preaction of $\rs$ on  Comod-$A(\delta)_d$.

To proceed we will need the following proposition describing the behaviour of the $A(\delta)_d$-comodules
$\ktr = \k_0$ and of $\k_{v_{i+1}}$ under $\ind{i}$.
\begin{theorem}\label{induction}
Suppose $1\le i \le n-1$ then
\begin{equation}
\label{ind0}
R^k \ind{i} \k_{0} \cong
\begin{cases}
\ktr, & k=0\\
0, & k\not=0
\end{cases}
\end{equation}
and
\begin{equation*}
R^k \ind{i} \k_{v_{i+1}} = 0,\ k\ge 0.
\end{equation*}
\end{theorem}
\begin{proof}
First we reduce the claim of the theorem to the case $(\alpha,\beta) =
(1,q)$. Then we will apply results of \cite{donkin1}.

Let $q:= \alpha\beta$. Consider the isomorphism of coalgebras $\varphi\colon A(n)\to A_{1,q}(n)\,$, $\varphi(c^\omega)= \alpha^{J(\omega)} c^\omega$,
defined in Theorem~\ref{bialgebra-isomorphism}
By Lemma~2.3 in
\cite{parshall}, $\varphi(d_{\alpha,\beta}) = d_{1,q}$. Therefore,
$\varphi$ can be extended to a map $\varphi'\colon A(n)_d\to A_{1,q}(n)_d $
by $\varphi' (d_{\alpha, \beta}^k x ) = d_{1,q}^k \varphi(x)$. It is shown in Theorem~2.4 of~\cite{parshall}, that
$\varphi'$ is an isomorphism of coalgebras.

The isomorphism $\varphi'$ induces an isomorphism of coalgebras $[\varphi']_\delta
\colon A_{\alpha,\beta}(\delta)_d \to A_{1,q}(\delta)_d$, $[\varphi']_\delta
([x]_\delta) = [\varphi'(x)]_\delta$ and $[\varphi']_{i} \colon
A_{\alpha,\beta}[i]_d \to A_{1,q}[i]_d$, $[\varphi']_{i} ([x]_{a[i]}) =
[\varphi'(x)]_{a[i]}$.

Therefore we get the following commutative diagram of coalgebras
\begin{equation*}
\xymatrix{
A_{\alpha,\beta}[i]_d \ar[r]_{\cong}^{[\varphi']_i} \ar[d]_{\pi_i} & A_{1,q}
[i]_d \ar[d]_{\pi_i} \\
A_{\alpha,\beta}(\delta)_d \ar[r]_{\cong}^{[\varphi']_\delta} & A_{1,q}(\delta)_d.
}
\end{equation*}
From this diagram it follows that we have to prove the theorem only for the case
$(\alpha,\beta) = (1,q)$, since the induction of comodules  involves only the coalgebra and the comodule structures.
 The case $(\alpha,\beta)=(1,q)$ was thoroughly studied in \cite{donkin1}, and
 both claims of the theorem now follow from Lemma~3.1 and Lemma~2.12 therein.
\end{proof}
\begin{corollary}
\label{eta}
The map $\eta\colon \ktr\to \ind{i} \k_0$, $1\mapsto 1\otimes 1$, is an
isomorphism of $A[i]_d$-comodules.
\end{corollary}
\begin{proof}
Consider the injective map $\eta'\colon \ktr \to \k_0 \otimes A[i]_d$
defined by $1\mapsto 1\otimes 1$. It is easy to see that the image of $\eta'$
lies in $\k_0\otimes^{A(\delta)_d} A[i]_d = \ind{i} \k_0$. Therefore we have the
monomorphism of $A[i]_d$-comodules $\eta\colon \ktr\to \ind{i}\k_0$. Since
$\dim \ind{i}\k_0$ is $1$, by Therorem~\ref{induction}, we get that $\eta$ is an
isomorphism.
\end{proof}
Notice that for any Hopf algebra $H$ and any $H$-comodule $N$,  $n\mapsto n \otimes 1$ defines an isomorphism between $N$ and $N\otimes \ktr$. Now let $N$ be an $A[i]_d$-comodule. We consider the following chain of isomorphisms
of $A[i]_d$-comodules
\begin{equation*}
N \stackrel{\cong}{\rightarrow} N\otimes \ktr \xrightarrow{1\otimes
\eta} N\otimes \ind{i}\k_0 \stackrel{\phi}{\rightarrow}
\ind{i}(N\otimes \k_0) \stackrel{\cong}{\rightarrow} \ind{i}{N},
\end{equation*}
where $\phi$ is the isomorphism~\eqref{Lambda1} and $\eta$ is defined in  Corollary~\ref{eta}.
Under this isomorphism we have for every $z\in N$
\begin{equation*}
z\mapsto z\otimes 1 \mapsto z \otimes 1 \otimes 1 \mapsto \sum z_{(0)} \otimes 1
\otimes z_{(1)} \mapsto \sum z_{(0)}\otimes  z_{(1)}.
\end{equation*}
Hence $\rho_N \colon N \to \ind{i}N$, $z \mapsto \sum z_{(0)} \otimes
z_{(1)}$ gives an isomorphism of $A[i]_d$-comodules.

We are now ready to define the natural isomorphism $$\tau_{ii}\colon F_i^2\rightarrow F_i, \,\,\,\mbox{all}\,\,\,1\leq i \leq n-1.$$
 Let $M\in \mbox{Comod-} A(\delta)_d$.
Then  $\ind{i}M$ is an $A[i]_d$-comodule with the comodule structure given by
\begin{equation*}
\sum_j z_j \otimes [x_j]_{a[i]} \mapsto \sum z_j \otimes
[x_{j,(1)}]_{a[i]} \otimes [x_{j,(2)}]_{a[i]},
\end{equation*}
where $z_j\in M$, and $x_j\in A(n)_d$. Therefore, we get for every $M\in
\mbox{Comod-}A(\delta)_d$ the isomorphism
\begin{equation}
\label{deltapim}
\begin{aligned}
\rho_{\ind{i}M} \colon \ind{i}M & \to \ind{i} F_i M\\[3ex]
\sum_j z_j \otimes [x_j]_{a[i]} &  \mapsto \sum z_j \otimes
[x_{j,(1)}]_{a[i]} \otimes [x_{j,(2)}]_{a[i]}.
\end{aligned}
\end{equation}
By restricting, we can consider $\rho_{\ind{i}M}$ as a
homomorphism of $A(\delta)_d$-comodules.  From the explicit expression of
$\rho_{\ind{i}M}$ it is obvious that the class of isomorphisms
$(\rho_{\ind{i}M})$ is a natural transformation of functors $F_i \to F_i^2$.
 We  define the natural
isomorphism $\tau_{ii} \colon F_i^2  \to F_i  $ as the inverse of
$(\rho_{\ind{i}M})$.

Before defining the isomorphisms $\tau_{ij}$ for $i<j$, we need to prove the
following theorem.

\begin{theorem}\label{comultiso}
Suppose $b\in \n^n$ satisfies $b\ge \delta$, the sequences $b$ and $b+v_l$ are
non-decreasing, and $b_{l-1} <b_l$. Then the map
\begin{equation}\label{eq:comultiso}
\begin{aligned}
A(b+v_l)_d & \to \ind{b_l}A(b)_d\\[3ex]
[x]_{b+v_l} & \mapsto \sum [x_{(1)}]_{b} \otimes [x_{(2)}]_{a[b_l]}
\end{aligned}
\end{equation}
is a well defined isomorphism of $A[b_l]_d$-comodules, and therefore  an isomorphism of $A(\delta)_d$-comodules.
\end{theorem}
\begin{proof}
By Proposition~\ref{seq2}, we have an exact sequence of $A(\delta)_d$-comodules
\begin{equation*}
0 \to A(b+v_l)_d\otimes \k_{v_{b_l+1}} \to A(b+v_l)_d  \to
A(b)_d \to 0.
\end{equation*}
Applying $\ind{b_l}$,  we get the exact sequence
\begin{equation*}
\ind{b_l} (A(b+v_l)_d\otimes \k_{v_{b_l+1}}) \to \ind{b_l} A(b+v_l)_d  \to
\ind{b_l} A(b)_d \to R^1 \ind{b_l} (A(b+v_l)_d\otimes \k_{v_{b_l+1}})
\end{equation*}
with the middle arrow given by
\begin{equation}\label{fblpi}
\begin{aligned}
\ind{b_l}A(b+v_l)_d & \to \ind{b_l}A(b)_d\\
\sum_k [x_k]_{b+v_l}\otimes [y_k]_{a[b_l]} & \mapsto \sum_k [x_k]_{b}\otimes [y_k]_{a[b_l]}.
\end{aligned}
\end{equation}
Note that since $b_{l-1} <b_l$ and $b + v_l$ is non-decreasing, the vector
$b+v_l$ does not have any component equal to $b_l$. Therefore, by
Remark~\ref{rem:first},
 $A(b+v_l)_d$ is an $A[b_l]_d$-comodule. Hence, by the tensor identity
(see Theorem~\ref{tensor}), we have
\begin{equation*}
R^i \ind{b_l} (A(b+v_l)_d\otimes \k_{v_{b_l+1}}) \cong A(b+v_l)_d\otimes
R^i \ind{b_l}\k_{v_{b_l+1}}.
\end{equation*}
But, by Theorem~\ref{induction},
\begin{equation*}
R^i \ind{b_l}\k_{v_{b_l+1}} =0,\ i\ge 0.
\end{equation*}
Therefore, \eqref{fblpi} is an isomorphism. Now,
using \eqref{Lambda1} for $H_1 = A[b_l]_d$, $H_2 = A(\delta)_d$, $M =
A(b+v_l)_d$, and $N = \k_0$, we get the isomorphism
\begin{equation}\label{abvld}
\begin{aligned}
A(b+v_l)_d \otimes \ind{b_l} \k_0 & \to \ind{b_l} A(b+v_l)_d \\
[x]_{b+v_l} \otimes 1 \otimes [y]_{a[b_l]} & \mapsto \sum [x_{(1)}]_{b+v_l}
\otimes [x_{(2)}]_{a[b_l]} [y]_{a[b_l]}.
\end{aligned}
\end{equation}
Recall that, in Corollary~\ref{eta}, we  defined the isomorphism of $A[b_l]_d$-comodules $\eta\colon
\ktr\to \ind{b_l} \k_0$ .

Composing $A(b+v_l)_d \otimes \eta$ with \eqref{abvld} and
\eqref{fblpi}, we get the isomorphism
\begin{equation}\label{final}
\begin{aligned}
A(b+v_l)_d \otimes \k_0 & \to \ind{b_l}A(b)_d \\
[x]_{b+v_l} \otimes 1 & \mapsto \sum [x_{(1)}]_{b} \otimes
[x_{(2)}]_{a[b_l]}.
\end{aligned}
\end{equation}
Composing this with the natural isomorphism $A(b+v_l)_d \to A(b+v_l)_d \otimes \k_0$,
we see that \eqref{eq:comultiso} is indeed a well-defined isomorphism of
$A[b_l]_d$-comodules.
\end{proof}
To define the isomorphisms $\tau_{ij}$, $i+2\leq j$, we proceed as follows.
Applying Theorem~\ref{comultiso} with $l=j$ and $l=i$, respectively, we get the isomorphisms
\begin{equation}
\label{adeltavivj}
\begin{aligned}
A(\delta+v_i + v_j)_d & \to \ind{j}A(\delta + v_i)_d
 \\ A(\delta+v_i + v_j)_d & \to \ind{i}A(\delta + v_j)_d .
\end{aligned}
\end{equation}
As  $\delta + v_i=a[i]$ and $\delta + v_j=a[j] $, composing the inverse of the first of these isomorphisms with the second one, we obtain the
isomorphism
\begin{equation*}
t_{ij} \colon \ind{j}A[i]_d  \to \ind{i}A[j]_d.
\end{equation*}
For $i+2\leq j$, we define $\tau_{ij}\colon F_j F_i  \to F_i F_j $,  by $(\tau_{ij})_M = M\otimes^{A(\delta)_d} t_{ij}$.
Clearly, the family $(\tau_{ij})$ is a natural transformation of functors.

Finally, we will define now the natural transformations $\tau_{i,i+1}$.
Applying Theorem~\ref{comultiso} with $l=i$, we get the isomorphisms
\begin{equation}
\begin{aligned}
A(\delta + 2v_i + v_{i+1})_d & \to \ind{i+1}A(\delta + v_i + v_{i+1})_d \\
A(\delta + v_i + v_{i+1})_d & \to\ind{i} A(\delta + v_{i+1})_d.
\end{aligned}
\end{equation}
Therefore, we have the isomorphism of $A(\delta)_d$-comodules
\begin{equation}
\label{adelta}
\begin{aligned}
A(\delta + 2v_i + v_{i+1})_d &\to A[i+1]_d \otimes^{A(\delta)_d}
A[i]_d \otimes^{A(\delta)_d} A[i+1]_d.
\end{aligned}
\end{equation}
Since comultiplication is coassoative on $A(n)_d$, the explicit formula for
\eqref{adelta} is given by
\begin{equation*}
[x]_{\delta + 2 v_i + v_{i+1}} \mapsto \sum [x_{(1)}]_{a[i+1]} \otimes
[x_{(2)}]_{a[i]} \otimes [x_{(3)}]_{a[i+1]}.
\end{equation*}
\newcommand{\m}[1]{\rho_{#1}}
\begin{proposition}
\label{prop:adelta2}
The map
\begin{equation}\label{adelta2}
\begin{aligned}
\m{\alpha,\beta} \colon A(\delta + 2v_i + v_{i+1})_d &\to A[i]_d \otimes^{A(\delta)_d}
A[i+1]_d \otimes^{A(\delta)_d} A[i]_d\\
[x]_{\delta + 2 v_i + v_{i+1}} &\mapsto \sum [x_{(1)}]_{a[i]} \otimes
[x_{(2)}]_{a[i+1]} \otimes [x_{(3)}]_{a[i]}
\end{aligned}
\end{equation}
is a well defined isomorphism of $A(\delta)_d$-comodules.
\end{proposition}
\begin{proof}
The idea of the proof
is to exhibit an isomorphism  that identifies \eqref{adelta2} with
\eqref{adelta}.

Without loss of generality we can assume that $q := (\alpha\beta)^{\frac12}\in
\k$. In fact, if \eqref{adelta2} is not an isomorphism, then it will  not be
an isomorphism upon field extension either.

	By Theorem~\ref{bialgebra-isomorphism} the map $\varphi\colon
A_{\alpha,\beta}(n) \to A_{q,q}(n)$, defined by $\varphi(c^\omega) =
(\alpha\beta^{-1})^{{\frac12}J(\omega)}c^\omega$, is an isomorphism of coalgebras. Using
Proposition~\ref{basis}, we see that $\varphi$ induces an
isomorphism of vector spaces $\varphi_a\colon A_{\alpha,\beta}(a) \to
A_{q,q}(a)$ for every non-decreasing  sequence $a \in \n^n$. If $a$ is of the form $(\lambda_1^{\lambda_1},
(\lambda_1+\lambda_2)^{\lambda_2}, \dots, n^{\lambda_r})$, then $A_{\alpha,\beta}(a)$ and $A_{q,q}(a)$  are coalgebras and we see that $\varphi_a$ is an isomorphism of coalgebras. This is the case of the sequences $\delta$, $a[i]$, $a[i+1]$ and
$b=\delta + 2v_i + v_{i+1}$. So we get the following commutative diagram
\begin{equation*}
\xymatrix@C1.9em{
&A_{\alpha,\beta}(b) \ar[rrr]^{\varphi_b} \ar[rd] \ar[ldd] \ar[ddd] &&&
A_{q,q}(b) \ar[rd] \ar[ldd] \ar[ddd] \\
&& A_{\alpha,\beta}[i] \ar[rrr]^(0.3){\varphi_{a[i]}}\ar[ldd]  &&& A_{q,q}[i]\ar[ldd] \\
A_{\alpha,\beta}[i+1] \ar[rrr]_(0.7){\varphi_{a[i+1]}} \ar[rd] &&& A_{q,q}[i+1] \ar[rd] \\
& A_{\alpha,\beta}(\delta) \ar[rrr]^-{\varphi_\delta} &&& A_{q,q}(\delta)
}
\end{equation*}
where all the maps are homomorphisms of coalgebras and the horizontal arrows are
isomorphisms.

From \cite[Lemma~2.3]{parshall}, we get that $\varphi_a [d_{\alpha,\beta}]_a =
[d_{q,q}]_a$. Thus the above diagram remains commutative upon localization.
This shows that we have
the following commutative diagram,
whose vertical arrows are isomorphisms
\begin{equation*}
\xymatrix{
A_{\alpha,\beta}(b)_d \ar[r]^-{\m{\alpha,\beta}} \ar[d]_-{\varphi_b} & A_{\alpha,\beta}[i]_d \otimes^{A_{\alpha,\beta}(\delta)_d} A_{\alpha,\beta}[i+1]_d \otimes^{A_{\alpha,\beta}(\delta)_d}
A_{\alpha,\beta}[i]_d \ar[d]^-{\varphi_{a[i]} \otimes \varphi_{a[i+1]} \otimes
\varphi_{a[i]}} \\
A_{q,q}(b)_d \ar[r]^-{{\m{q,q}}}   & A_{q,q}[i]_d \otimes^{A_{q,q}(\delta)_d} A_{q,q}[i+1]_d \otimes^{A_{q,q}(\delta)_d} A_{q,q}[i]_d.
}
\end{equation*}
Therefore it is enough
to prove the proposition in the case
$(\alpha,\beta) = (q,q)$. In this case we can use the results of Parshall and Wang in~\cite{parshall:book}.
\newcommand{\h}{h}
Proposition~3.7.1(3) of that work says that  the
map $\h$ sending $c_{is}$ to $c_{n+1-s,n+1-i}$ extends to an
anti-automorphsims of $A_{q,q}(n)$ considered  both as a coalgebra and an algebra.
It is not difficult to check that
\begin{equation*}
\begin{aligned}
&\h (I(\delta))  = I(\delta)\\
&\h (I(a[i])) = I (a[n-i])\\
&\h(I(a[i+1])) = I(a[n-i-1])\\
&\h (I(b)) = I(\delta + 2 v_{n-i-1} + v_{n-i}).
\end{aligned}
\end{equation*}
Thus if $b'=\delta + 2v_{n-i-1} + v_{n-i}$,  we get the
commutative diagram
\begin{equation*}
\xymatrix@C0.9em{
&A_{q,q}(b) \ar[rrr]^{\h_b} \ar[rd] \ar[ldd] \ar[ddd] &&&
A_{q,q}(b')^{op} \ar[rd] \ar[ldd] \ar[ddd] \\
&& A_{q,q}[i] \ar[rrr]^(0.3){\h_{a[i]}}\ar[ldd]  &&& A_{q,q}[n-i]^{op}\ar[ldd] \\
A_{q,q}[i+1] \ar[rrr]_(0.6){\h_{a[i+1]}} \ar[rd] &&& A_{q,q}[n-i-1]^{op} \ar[rd] \\
& A_{q,q}(\delta) \ar[rrr]^-{\h_\delta} &&& A_{q,q}(\delta)^{op}
}
\end{equation*}
where all the horizontal arrows are isomorphisms of coalgebras, and all slanted
arrows are natural projections preserving comultiplication.
It is shown in \cite[Lemma~4.2.3]{parshall:book}, that $\h(d_{q,q}) =
d_{q,q}$. Therefore, we have a similar diagram with all the bialgebras replaced by
their localizations with respect to $[d]_a$, for a suitable $a$.
We get then the commutative diagram
\begin{equation*}
\xymatrix{
A_{q,q}(b)_d \ar[r]^-{\m{q,q}} \ar[d]_-{\h_b} & A_{q,q}[i]_d \otimes^{A_{q,q}(\delta)_d} A_{q,q}[i+1]_d \otimes^{A_{q,q}(\delta)_d}
A_{q,q}[i]_d \ar[d]^-{\h_{a[i]} \otimes \h_{a[i+1]} \otimes
\h_{a[i]}} \\
A_{q,q}(b')_d^{op} \ar[r]^-{{\m{q,q}'}}   & A_{q,q}[n-i]_d^{op}
\otimes^{A_{q,q}(\delta)_d^{op}} A_{q,q}[n-i-1]_d^{op}
\otimes^{A_{q,q}(\delta)_d^{op}} A_{q,q}[n-i]_d^{op}.
}
\end{equation*}
whose vertical arrows are isomorphisms
and the map $\m{q,q}'$ is given by
\begin{equation*}
\begin{aligned}
\m{q,q}'\colon [x]_{b'} & {\mapsto} \sum [x_{(3)}]_{a[n-i]} \otimes  [x_{(2)}]_{a[n-i-1]} \otimes
[x_{(1)}]_{a[n-i]}.
\end{aligned}
\end{equation*}
Thus it is enough to prove that $\m{q,q}'$
is an isomorphism.
It follows, from Remark~\ref{associativity}, that the linear isomorphism
 $A_{q,q}[n-i]_d^{op} \otimes
A_{q,q}[n-i-1]_d^{op} \otimes A_{q,q}[n-i]_d^{op}$ $\to$
$A_{q,q}[n-i]_d \otimes
A_{q,q}[n-i-1]_d \otimes A_{q,q}[n-i]_d$
given by
\begin{equation*}
\begin{aligned}
 a_1 \otimes a_2 \otimes a_3 &\mapsto a_3 \otimes a_2 \otimes a_1
\end{aligned}
\end{equation*}
induces a linear isomorphism $\nu$ between $A_{q,q}[n-i]_d^{op} \otimes^{A_{q,q}(\delta)_d^{op}}
A_{q,q}[n-i-1]_d^{op} \otimes^{A_{q,q}(\delta)_d^{op}}A_{q,q}[n-i]_d^{op}$ and $ A_{q,q}[n-i]_d
\otimes^{A_{q,q}(\delta)_d} A_{q,q}[n-i-1]_d
\otimes^{A_{q,q}(\delta)_d} A_{q,q}[n-i]_d$

Therefore we get the commutative diagram
\begin{equation*}
\xymatrix{
A_{q,q}(b')_d^{op} \ar[r]^-{{\m{q,q}'}} \ar[d]^-{\id}   & A_{q,q}[n-i]_d^{op}
\otimes^{A_{q,q}(\delta)_d^{op}} A_{q,q}[n-i-1]_d^{op}
\otimes^{A_{q,q}(\delta)_d^{op}} A_{q,q}[n-i]_d^{op} \ar[d]_-{\nu}\\
A_{q,q}(b')_d \ar[r]^-{{\m{q,q}''}}   & A_{q,q}[n-i]_d
\otimes^{A_{q,q}(\delta)_d} A_{q,q}[n-i-1]_d
\otimes^{A_{q,q}(\delta)_d} A_{q,q}[n-i]_d
,
}
\end{equation*}
where $\m{q,q}''$ is the isomorphism~\eqref{adelta}.
This shows that $\m{q,q}'$ is an isomorphism, and the result follows.

\end{proof}
We define the map $t_{i,i+1}$ as the composition of the inverse
of~\eqref{adelta} followed by~\eqref{adelta2}.
Therefore
\begin{equation*}
t_{i,i+1}\colon A[i+1]_d \otimes^{A(\delta)_d}
A[i]_d \otimes^{A(\delta)_d} A[i+1]_d \to A[i]_d \otimes^{A(\delta)_d}
A[i+1]_d \otimes^{A(\delta)_d} A[i]_d.
\end{equation*}

Now the natural transformations $\tau_{i,i+1}$ are defined by $(\tau_{i,i+1})_M = M\otimes^{A(\delta)_d} t_{i,i+1}$, i.e.,
\begin{equation*}
\begin{aligned}
F_{i+1}F_i F_{i+1} M &\to F_i F_{i+1} F_i M\\
\sum_k m_k \otimes w_k
&\mapsto \sum_k m_k \otimes t_{i,i+1} (w_k),
\end{aligned}
\end{equation*}
all  $m_k \in M, \,\, w_k \in  A[i+1]_d \otimes^{A(\delta)_d} A[i]_d \otimes^{A(\delta)_d} A[i+1]_d .$
\section{The commutativity of the preaction diagrams }\label{section6}
\label{sixth}
We will show now that the natural isomorphisms $\tau_{ij}$,
defined in the previous section, satisfy all the necessary relations
 so that $(F_i, 1\le i\le n-1;
\tau_{ij}, 1\le i\le j\le n-1)$ is a preaction (in the sense of
Section~\ref{second}) of $\rs$ on the category
Comod-$A(\delta)_{d}$.

We start by  describing the notation used in the diagrams below. First of all, note that if $M$ is an $A(\delta)_d$-comodule, then
\begin{equation*}
F_{i_k} \dots F_{i_1} M = M \otimes^{A(\delta)_d} A[i_1]_d \otimes^{A(\delta)_d}
\dots \otimes^{A(\delta)_d} A[i_k]_d.
\end{equation*}
Suppose that $\lambda= (\lambda_1, \cdots, \lambda _m)$ is a composition of $n$, and $b=(\lambda_1^{\lambda_1},
(\lambda_1+\lambda_2)^{\lambda_2}, \dots, n^{\lambda_m}).$
Then, using the coassosiativity of the comultiplication on $A(n)_d$, we get the map
\begin{equation}
\label{abd}
\begin{aligned}
A(b)_d&  \to \underbrace{A(b)_d \otimes^{A(\delta)_d} \dots
\otimes^{A(\delta)_d} A(b)_d}_{k \mbox{ times}}\\[3ex]
[x]_b &\mapsto \sum [x_{(1)}]_b \otimes \dots \otimes [x_{(k)}]_b.
\end{aligned}
\end{equation}
Suppose now  that $b^{(1)}$,\dots, $b^{(k)} \in \n^n$ satisfy
$\delta \le b^{(i)} \le b$, for $1\le i\le k$. Composing \eqref{abd} with the canonical
projections $A(b)_d \to A(b^{(i)})_d$, we get the map
\begin{equation*}
\begin{aligned}
\rho_{b;b^{(1)},\dots,b^{(k)}} \colon   A(b)_d & \to
 A(b^{(1)})_d \otimes^{A(\delta)_d} \dots
\otimes^{A(\delta)_d} A(b^{(k)})_d\\[3ex]
 [x]_b &\mapsto \sum [x_{(1)}]_{b^{(1)}} \otimes \dots \otimes
[x_{(k)}]_{b^{(k)}}.
\end{aligned}
\end{equation*}

\begin{remark}
\label{rem:isos}
In the case $k=2$, $b=b^{(1)} = b^{(2)}= a[i]$, we  recover $(\tau_{ii}^{-1})_M= M\otimes^{A(\delta)_d}\rho_{b;b^{(1)},b^{(2)}}$. For
$k=2$, $b = \delta + v_i + v_j$, $b^{(1)}=a[i]$ ($b^{(1)} = a[j]$),  and  $b^{(2)} =
a[j]$ ($b^{(2)} = a[i]$),  we get isomorphisms, since~\eqref{adeltavivj} are
isomorphisms. For $k=3$, $b = \delta + 2v_i + v_{i+1}$, $b^{(1)}= b^{(3)} = v_i$,
$b^{(2)} = v_{i+1}$, we get that $ \rho_{b;b^{(1)},b^{(2)},b^{(3)}}$ is an isomorphism by
Proposition~\ref{prop:adelta2}.
For $k=3$, $b= \delta + 2v_i + v_{i+1}$, $b^{(1)} = b^{(3)}= v_{i+1}$, $b^{(2)} = v_i$, we
get that $\rho_{b;b^{(1)},b^{(2)},b^{(3)}}$ is an isomorphism, since~\eqref{adelta} is an
isomorphism.
\end{remark}

In the diagrams below we will skip $M$ and write:
\begin{enumerate}[i)]
\item $i_1^{\alpha_1} \dots
i_l^{\alpha_l}$ for $A(\delta + \sum_{k=1}^l \alpha_k v_{i_k} )_d$, where $1\le i_1< \dots
<i_l \le n-1 $, and $1\le \alpha_k\le n-i_k$;
\item  dot ``$.$'' for $\otimes^{A(\delta)_d}$;
\item $\rho_{k}$ for $\rho_{b;b^{(1)},\dots,b^{(k)}}$, and $\rho$ for
$\rho_{2}$.
\end{enumerate}
For example, \begin{equation*}
\xymatrix@C0.5cm{ (i+1).i .(i+1) \ar@{<-}[r]^-{\rho_3} & i^2(i+1)
 \ar[r]^-{\rho_3} & i.(i+1).i\,,
&   i .j \ar@{<-}[r]^-{\rho} & ij
 \ar[r]^-{\rho} & j.i\,,}
\end{equation*}
 \begin{equation*}
\xymatrix@C0.6cm{ (i+1).i .(i+1).i \ar@{<-}[r]^-{\rho_3.i} & i^2(i+1).i
 \ar[r]^-{\rho_3.i} & i.(i+1).i.i\,,& \!\! i. i .j \ar@{<-}[r]^-{i.\rho} &i. ij
 \ar[r]^-{i.\rho} &i. j.i\,,}\end{equation*} denote, respectively, $\tau_{i,i+1}$,  $\tau_{i,j}$, $F_i \tau_{i,i+1}$ and $\tau_{i,j}F_i$.

Note that all the diagrams below are commutative, since comultiplication in $A(n)$
is coassociative. Moreover, the maps at the boundaries are isomorphisms
by~Remark~\ref{rem:isos}.

We have to check that two paths going from the
upper-left vertex to the down-right vertex produce equal maps.
For this it is enough to check that all the maps which are not at the boundary are also
isomorphisms.

\renewcommand{\comult}{\rho}
In the diagram
\begin{equation*}
\xymatrix{
i.i.i & i.i \ar[l]_-{i.\comult}\\
i.i \ar[u]^{\comult.i} & i \ar[u]_-{\comult} \ar[l]_-{\comult}
}
\end{equation*}
there is nothing to check since there are no arrows except the boundary ones.

In the diagrams
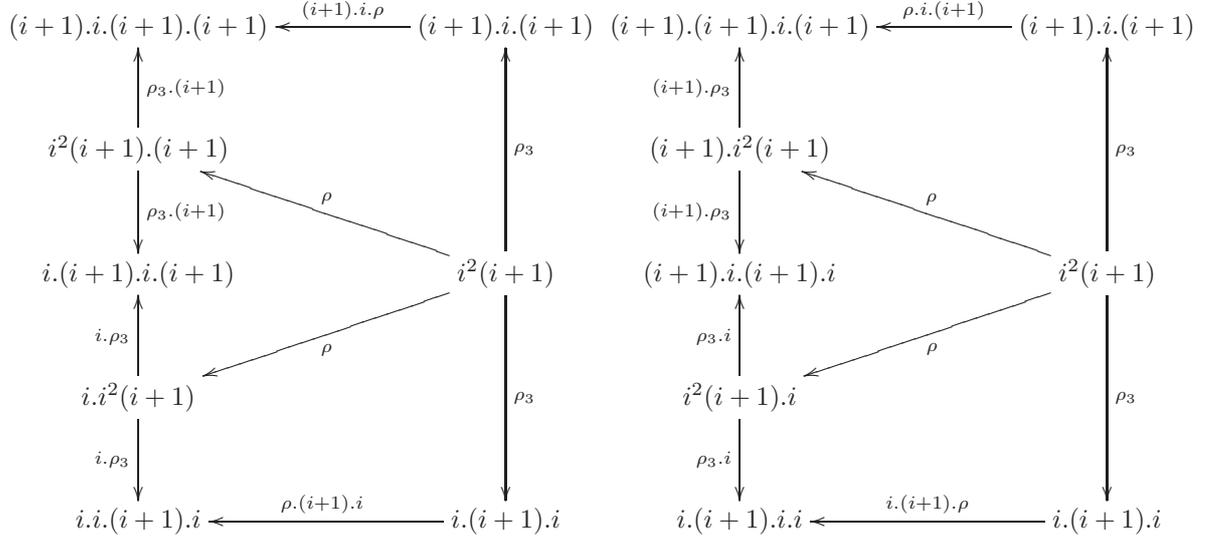
\begin{equation*}
\xymatrix@R7ex@C5em{
(i+1).i.(i+1).(i+1)  & (i+1).i.(i+1) \ar[l]_-{(i+1).i.\comult} \\
i^2(i+1).(i+1) \ar[u]_{\comult_3.(i+1)} \ar[d]^-{\comult_3.(i+1)}   \\
i.(i+1).i.(i+1) & i^2(i+1) \ar[uu]_-{\comult_3} \ar[dd]^{\comult_3}
\ar[lu]_{\comult} \ar[dl]^{\comult} \\ i.i^2(i+1) \ar[u]^{i.\comult_3} \ar[d]_{i.\comult_3} \\
i.i.(i+1).i & i.(i+1).i \ar[l]_{\comult.(i+1).i}
}
\xymatrix@R7ex@C5em{
(i+1).(i+1).i.(i+1)  & (i+1).i.(i+1)\ar[l]_-{\comult.i.(i+1)}\\
(i+1).i^2(i+1)\ar[u]^{(i+1).\comult_3}\ar[d]_{(i+1).\comult_3} &
\\
(i+1).i.(i+1).i & i^2(i+1)\ar[uu]_{\comult_3}\ar[dd]^{\comult_3} \ar[lu]_-{\comult}
\ar[ld]^-{\comult}\\
i^2(i+1).i \ar[u]^{\comult_3.i} \ar[d]_{\comult_3.i} \\
i.(i+1).i.i  &  i.(i+1).i\ar[l]_-{i.(i+1).\comult}
}
\end{equation*}
the invertibility of non-boundary maps follows from the commutativity of the
upper and lower trapezoids.

In the diagrams
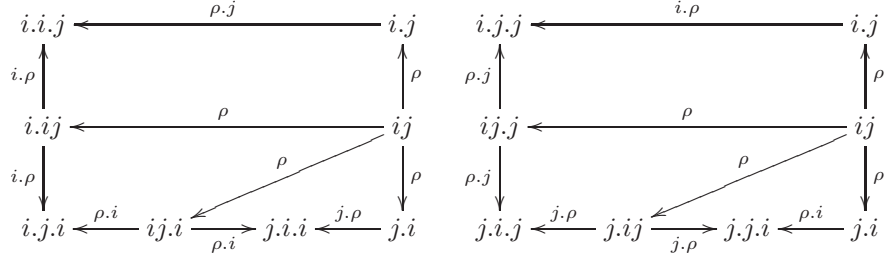
\begin{equation*}
\xymatrix{
i.i.j &&& i.j \ar[lll]_{\comult.j} \\
i.ij \ar[u]^{i.\comult} \ar[d]_{i.\comult}&& & ij \ar[u]_{\comult} \ar[d]^{\comult}
\ar[lll]_{\comult} \ar[dll]_-{\comult}
\\
i.j.i & ij.i \ar[l]_-{\comult.i} \ar[r]_-{\comult.i} & j.i.i & j.i \ar[l]_-{j.\comult}
} \quad
\xymatrix{
i.j.j &&& i.j \ar[lll]_-{i.\comult}\\
ij.j \ar[u]^{\comult.j} \ar[d]_{\comult.j} &&& ij \ar[u]_{\comult}
\ar[d]^{\comult} \ar[lll]_-{\comult} \ar[lld]_-{\comult} \\
j.i.j & j.ij \ar[l]_-{j.\comult} \ar[r]_-{j.\comult} & j.j.i & j.i
\ar[l]_-{\comult.i}
}
\end{equation*}
the invertibility of non-boundary arrows follows from the commutativity of the
upper rectangles and  the commutativity of the lower-down triangles.

It is not difficult to conclude, by a recursive argument, that in the next diagrams
 it is enough to check that one of the radial arrows is invertible to
conclude that all the radial arrows are isomorphisms.

In the diagram
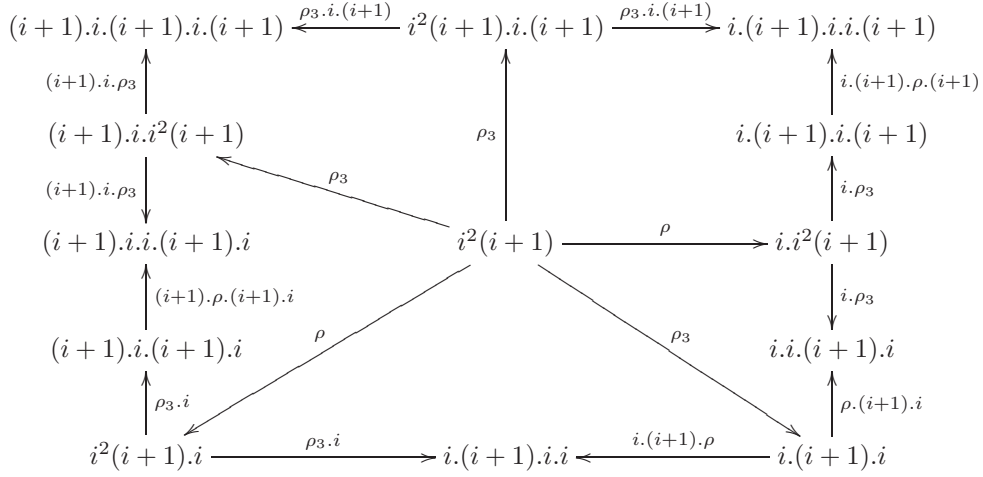
\begin{equation*}
\xymatrix@C4em{
(i+1).i.(i+1).i.(i+1) & i^2(i+1).i.(i+1) \ar[l]_-{\comult_3.i.(i+1)}
\ar[r]^-{\comult_3.i.(i+1)} & i.(i+1).i.i.(i+1)  \\
(i+1).i.i^2(i+1) \ar[u]^{(i+1).i.\comult_3} \ar[d]_{(i+1).i.\comult_3} &
  &
i.(i+1).i.(i+1) \ar[u]_-{i.(i+1).\comult .(i+1)} \\
(i+1).i.i.(i+1).i & i^2(i+1) \ar[ul]_-{\comult_3} \ar[uu]^-{\comult_3}
\ar[r]^-{\comult} \ar[rdd]^-{\comult_3} \ar[ldd]_-{\comult}    & i.i^2(i+1) \ar[u]_-{i.\comult_3}
\ar[d]^-{i.\comult_3} \\
(i+1).i.(i+1).i \ar[u]_{(i+1).\comult.(i+1).i} &  & i.i.(i+1).i\\
i^2(i+1).i \ar[u]_-{\comult_3.i} \ar[r]^-{\comult_3.i} & i.(i+1).i.i &
i.(i+1).i \ar[l]_-{i.(i+1).\comult} \ar[u]_-{\comult.(i+1).i}
}
\end{equation*}
the $5$ o'clock map $\rho_3 \colon i^2(i+1) \to i.(i+1).i$ is invertible, by
Remark~\ref{rem:isos}.

In the diagram
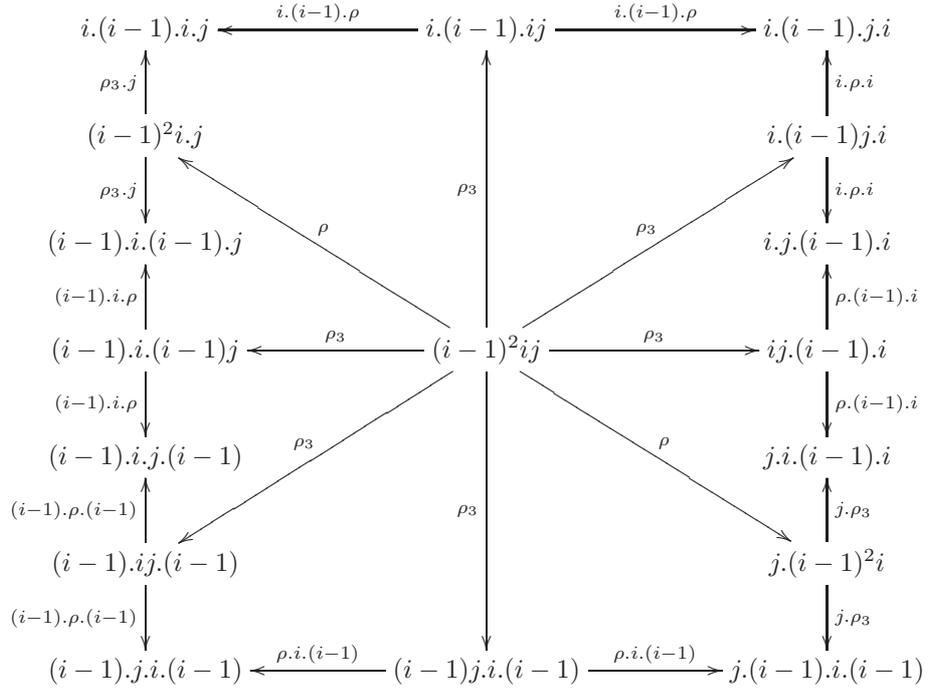
\begin{equation*}
\xymatrix@C5em{
i.(i-1).i.j  & i.(i-1).ij \ar[l]_-{i.(i-1).\comult}
\ar[r]^-{i.(i-1).\comult} & i.(i-1).j.i   \\
 (i-1)^2i.j \ar[u]^-{\comult_3.j} \ar[d]_-{\comult_3.j} &&
i.(i-1)j.i\ar[u]_-{i.\comult.i} \ar[d]^-{i.\comult.i}\\
(i-1).i.(i-1).j & & i.j.(i-1).i \\
(i-1).i.(i-1)j \ar[u]^-{(i-1).i.\comult} \ar[d]_-{(i-1).i.\comult}
 &  (i-1)^2ij
\ar[uuu]^-{\comult_3} \ar[ruu]^-{\comult_3} \ar[r]^-{\comult_3}
\ar[rdd]^-{\comult} \ar[ddd]_-{\comult_3} \ar[ldd]_-{\comult_3}
\ar[l]_-{\comult_3} \ar[luu]_-{\comult}
 &
ij.(i-1).i \ar[u]_-{\comult.(i-1).i} \ar[d]^-{\comult.(i-1).i}
\\
(i-1).i.j.(i-1) && j.i.(i-1).i \\
(i-1).ij.(i-1) \ar[u]^-{(i-1).\comult.(i-1)} \ar[d]_-{(i-1).\comult.(i-1)}  &&
j.(i-1)^2i \ar[u]_-{j.\comult_3} \ar[d]^-{j.\comult_3} \\
(i-1).j.i.(i-1) &
(i-1)j.i.(i-1) \ar[l]_-{\comult.i.(i-1)} \ar[r]^-{\comult.i.(i-1)} &
j.(i-1).i.(i-1)
}
\end{equation*}
the $10$ o'clock map $\comult\colon (i-1)^2ij \to
(i-1)^2i.j$ is an isomorphism, by
Theorem~\ref{comultiso}.
\begin{figure}
\caption{\label{lastdiagram}
}
\begin{equation*}
\xymatrix@C5em{
k.j.k.i.j.k & k.j.ik.j.k \ar[l]_-{k.j.\rho.j.k} \ar[r]^-{k.j.\rho.j.k}  &
k.j.i.k.j.k & k.j.i.j^2k \ar[l]_-{k.j.i.\rho_3} \ar[d]_-{k.j.i.\rho_3} \\
j^2k.i.j.k \ar[u]^-{\rho_3.i.j.k}
\ar[d]_-{\rho_3.i.j.k}  & & &
k.j.i.j.k.j\\
j.k.j.i.j.k    &&&
k.i^2j.k.j \ar[u]_-{k.\rho_3.k.j}
\ar[d]^-{k.\rho_3.k.j}  \\
j.k.i^2j.k \ar[u]^-{j.k.\rho_3.k}
\ar[d]_-{j.k.\rho_3.k} &&& k.i.j.i.k.j \\
j.k.i.j.i.k &&& k.i.j.ik.j \ar[d]^-{k.i.j.\rho.j}
\ar[u]_-{k.i.j.\rho.j} \\
j.k.i.j.ik \ar[d]_-{j.k.i.j.\rho}
\ar[u]^-{j.k.i.j.\rho}  &
i^3j^2k
\ar[uuuuu]^-{\rho_5} \ar@/^7ex/[uuuuurr]^{\rho_4}
\ar[uuurr]^-{\rho_4}  \ar[urr]^-{\rho_5}
\ar[drr]^-{\rho_5} \ar[dddrr]^-{\rho_4}
\ar@/_7ex/[dddddrr]^-{\rho_4} \ar@/_3ex/[ddddddr]^-{\rho_5}
\ar@/^5ex/[ddddddl]^-{\rho_4} \ar[ddddl]_-{\rho_4}
\ar[ddl]_-{\rho_5} \ar[l]_-{\rho_5}
\ar[uul]_-{\rho_4} \ar[uuuul]_-{\rho_4}
&& k.i.j.k.i.j \\
j.k.i.j.k.i &&& ik.j.k.i.j \ar[d]^-{\rho.j.k.i.j}
\ar[u]_-{\rho.j.k.i.j}  \\
j.ik.j.k.i \ar[d]_-{j.\rho.j.k.i}
\ar[u]^-{j.\rho.j.k.i}  &&& i.k.j.k.i.j \\
j.i.k.j.k.i  &&& i.j^2k.i.j \ar[d]_-{i.\rho_3.i.j}
\ar[u]^-{i.\rho_3.i.j} \\
j.i.j^2k.i \ar[d]_-{j.i.\rho_3.i}
\ar[u]^-{j.i.\rho_3.i}  &&& i.j.k.j.i.j \\
j.i.j.k.j.i &&& i.j.k.i^2j \ar[d]_-{i.j.k.\rho_3}
\ar[u]^-{i.j.k.\rho_3} \\
i^2j.k.j.i \ar[u]^-{\rho_3.k.j.i} \ar[r]^-{\rho_3.k.j.i} & i.j.i.k.j.i &
i.j.ik.j.i \ar[r]^-{i.j.\rho.j.i}
\ar[l]_-{i.j.\rho.j.i} & i.j.k.i.j.i
}
\end{equation*}
\end{figure}

In the diagram
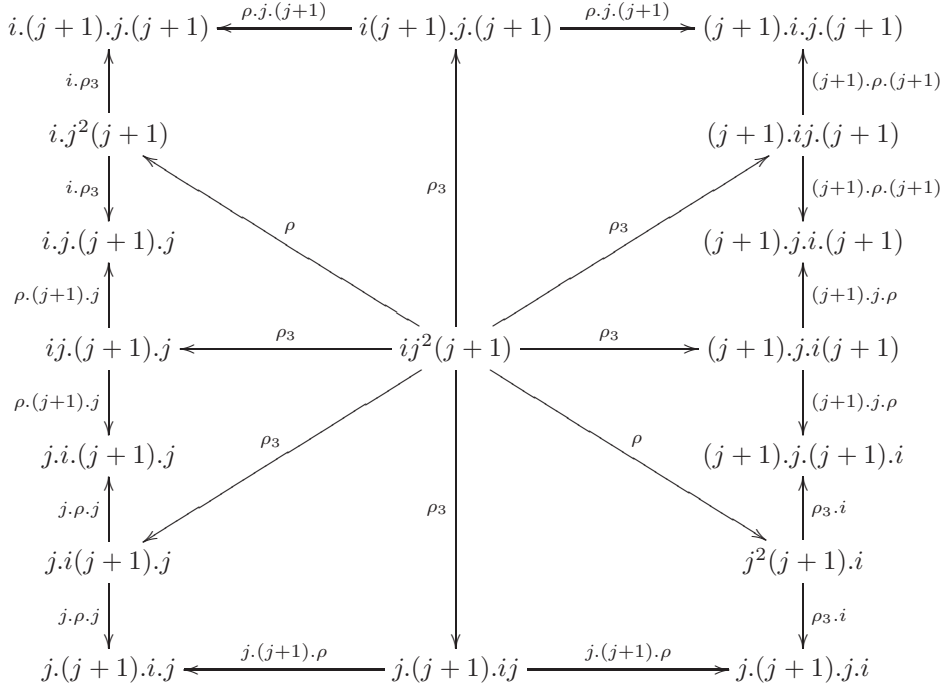
\begin{equation*}
\xymatrix@C5em{
i.(j+1).j.(j+1) & i(j+1).j.(j+1) \ar[l]_-{\comult.j.(j+1)} \ar[r]^-{\comult.j.(j+1)} &
(j+1).i.j.(j+1) \\
i.j^2(j+1) \ar[u]^-{i.\comult_3} \ar[d]_-{i.\comult_3} &&
(j+1).ij.(j+1) \ar[u]_-{(j+1).\comult.(j+1)} \ar[d]^-{(j+1).\comult.(j+1)}\\
i.j.(j+1).j && (j+1).j.i.(j+1) \\
ij.(j+1).j \ar[u]^-{\comult.(j+1).j} \ar[d]_-{\comult.(j+1).j} &
ij^2(j+1)
\ar[uuu]^-{\comult_3} \ar[ruu]^-{\comult_3} \ar[r]^-{\comult_3}
\ar[rdd]^-{\comult} \ar[ddd]_-{\comult_3} \ar[ldd]_-{\comult_3}
\ar[l]_-{\comult_3} \ar[luu]_-{\comult}
&
(j+1).j.i(j+1) \ar[u]_-{(j+1).j.\comult} \ar[d]^-{(j+1).j.\comult} \\
j.i.(j+1).j && (j+1).j.(j+1).i\\
j.i(j+1).j \ar[u]^-{j.\comult.j} \ar[d]_-{j.\comult.j} &&
j^2(j+1).i \ar[u]_-{\comult_3.i} \ar[d]^-{\comult_3.i} \\
j.(j+1).i.j &
j.(j+1).ij \ar[l]_-{j.(j+1).\comult} \ar[r]^-{j.(j+1).\comult} &
j.(j+1).j.i
}
\end{equation*}

the $4$ o'clock map $\rho\colon ij^2(j+1) \to j^2(j+1).i$ is an isomorphism, by
Theorem~\ref{comultiso}.

In the diagram
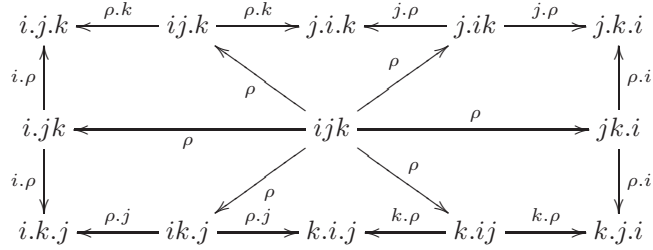
\begin{equation*}
\xymatrix@C3em{
i.j.k & ij.k \ar[l]_-{\comult.k} \ar[r]^-{\comult.k} & j.i.k & j.ik
\ar[l]_-{j.\comult} \ar[r]^-{j.\comult} & j.k.i  \\
i.jk \ar[u]^-{i.\comult} \ar[d]_-{i.\comult} &&
ijk \ar[lu]^{\comult} \ar[ru]^{\comult} \ar[rr]^{\comult} \ar[ll]^{\comult} \ar[ld]^{\comult} \ar[rd]^{\comult}
&&
 jk.i \ar[u]_-{\comult.i} \ar[d]^-{\comult.i} \\
i.k.j &
ik.j \ar[l]_-{\comult.j} \ar[r]^-{\comult.j}
&
k.i.j & k.ij \ar[l]_-{k.\comult} \ar[r]^-{k.\comult}&
k.j.i
 }
\end{equation*}
for example, the map $\rho\colon ijk \to ij.k$ is an isomorphism, by
Theorem~\ref{comultiso}.

In the   diagram  depicted in Figure~\ref{lastdiagram}, we write $j=i+1$ and  $k=i+2$.
In this diagram the $11$ o'clock map  $\rho_4 \colon i^3(i+1)^2(i+2) \to j^2k.i.j.k $
is an isomorphism, since it is  the following composition of isomorphisms defined
in Theorem~\ref{comultiso}
\begin{equation*}
i^3j^2k \to i^2j^2k.k \to ij^2k.j.k \to j^2k.i.j.k \quad .
\end{equation*}

This concludes the proof that the collection of functors $F_i$, $1\le i\le n-1$,
and of natural isomorphisms $\tau_{ij}$, $1\le i\le j\le n-1$, defines a preaction
of $\rs$ on Comod-$A(\delta)_d$.

\section{A
(pre)action of $\rs$ on $S_{\alpha,\beta}^-(n,r)$\mbox{-Mod}}
\label{seventh}

In this section we show that the preaction of $\rs$ on Comod-$A(\delta)_d$ induces a preaction (and so an action) of $\rs$ on the category of $ S^-(n,r)$-modules, where   $ S^-(n,r)= S_{\alpha,\beta}^-(n,r)$ is the  quantum negative  Borel-Schur algebra.

We prove first that the preaction of $\rs$ on Comod-$A(\delta)_d$ can be
restricted to Comod-$A(\delta)$.

For each $1\le i \le n-1$, define $F'_i  \colon
\mbox{Comod-}A(\delta) \to \mbox{Comod-}A(\delta)$ by
$$
M\mapsto M
\otimes^{A(\delta)} A[i].
$$ Let $\psi\colon A(\delta) \to A(\delta)_d$ be the canonical inclusion. Then we have the associated restriction functor $\psi_\circ  \colon
\mbox{Comod-}A(\delta) \to \mbox{Comod-}A(\delta)_d$.
\begin{proposition}
 The
inclusion $\psi_i \colon A[i] \to A[i]_d$
induces a natural isomorphism $\psi_\circ F'_i\to F_i \psi_\circ $.
\end{proposition}
\begin{proof}
Let $M \in \mbox{Comod-}A(\delta)$. Then the natural transformation in question is given by
\begin{equation*}
\begin{aligned}
M \otimes^{A(\delta)} A[i]  \to \psi_\circ M \otimes^{A(\delta)_d} A[i]_d\\
\sum_j  m_j \otimes x_j \mapsto  \sum m_j \otimes \psi_i(x_j).
\end{aligned}
\end{equation*}
Since  we have $M \cong M
\otimes^{A(\delta)} A(\delta)$, and the cotensor product is associative, to prove the proposition  it is enough
to show that
\begin{equation*}
\begin{aligned}
A(\delta) \otimes^{A(\delta)} A[i]  \to \psi_\circ A(\delta) \otimes^{A(\delta)_d} A[i]_d\\
\sum_j  z_j \otimes x_j \mapsto  \sum z_j \otimes \psi_i(x_j)
\end{aligned}
\end{equation*}
is an isomorphism.
Precomposing this with the isomorphism
$\,
A[i]  \to A(\delta)\otimes^{A(\delta)} A[i]\,$,
$[x]_{a[i]}   \mapsto \sum [x_{(1)}]_{\delta} \otimes [x_{(2)}]_{a[i]}
\,,$ we get
the map
\begin{equation}
\label{ai}
\begin{aligned}
A[i] & \to A(\delta)\otimes^{A(\delta)_d} A[i]_d\\
[x]_{a[i]}  & \mapsto \sum [x_{(1)}]_{\delta} \otimes [x_{(2)}]_{a[i]}.
\end{aligned}
\end{equation}
Thus all that is left  to check is  that \eqref{ai} is an isomorphism. For this, consider the
exact sequence~\eqref{sequence}, for $b=a[i]$ and $l=i$,
\begin{equation*}
0 \rightarrow A[i] \otimes \k_{v_{i+1}} \rightarrow A[i] \rightarrow A(\delta)
\rightarrow 0.
\end{equation*}
It can also be considered  as a sequence of $A(\delta)_d$-comodules. Proceeding as in the
proof of Theorem~\ref{comultiso} (with $b=\delta$ and $l=i$),  using Theorem~\ref{induction} and
Theorem~\ref{tensor}, we see that \eqref{ai} is an isomorphism of $A(\delta)_d$-comodules.
\end{proof}
Note that,  since $\psi\colon A(\delta) \to A(\delta)_d$ is a monomorphism of
coalgebras over a field, the functor $\psi_\circ$  is full and faithful.
Therefore, for any $M\in \mbox{Comod-}A(\delta)$, we have an
 isomorphism
\begin{equation*}
\mbox{Comod-}A(\delta) ( (F'_i)^2 M, F'_i M) \xrightarrow{\cong}
\mbox{Comod-}A(\delta)_d (F_i^2 \psi_\circ M, F_i \psi_\circ M),
\end{equation*}
for every $1\le i\le n-1$. Hence we can define $(\tau'_{ii})_M$ as the map that
corresponds to $(\tau_{ii})_M$ under this isomorphisms. It is clear that
$\tau'_{ii}$
is a natural transformation from $(F'_i)^2$ to $F'_i$. Similarly, one can define the
natural transformations $\tau'_{i,j}$ for $i<j$. Since $(F,\tau)$ is a preaction
on Comod-$A(\delta)_d$, we get that $(F',\tau')$ is a preaction on
Comod-$A(\delta)$.

Let $r$ be a natural number. Then the subset $A(\delta;r)$ of $r$-homogeneous
elements in $A(\delta)$ is a finite dimensional subcoalgebra of $A(\delta)$.
Similarly, the set $A(a[i];r)$ of $r$-homogeneous elements in $A[i]$ is a finite
dimensional subcoalgebra of $A[i]$.
Let $M$ be an $A(\delta;r)$-comodule.
Then from the definition of the cotensor product we get
\begin{equation*}
M \otimes^{A(\delta)} A[i] =  M\otimes^{A(\delta;r)} A(a[i];r).
\end{equation*}
Thus $F_i'M$ is  an $A(\delta;r)$-comodule. Hence the preaction
$(F',\tau')$ defines a preaction of $\rs$ on Comod-$A(\delta;r)$.

 As it is well known, see e.g.~\cite{donkin1}, \cite{DSY}, the
 associative algebra $S^-(n,r)=S_{\alpha,\beta}^-(n,r)$ dual to $A(\delta;r)$ is called the (negative) quantised \emph{Borel-Schur} algebra.
As usual, we have a canonical equivalence between the categories
$S^-(n,r)$-Mod and Comod-$A(\delta;r)$.
Therefore we get that $(F',\tau')$ induces an
action of $\rs$ on $S^-(n,r)$-Mod.

\section{Examples}
\label{sec:examples}
In this section we
consider some explicit examples of the application of the functors $F_w$ to $A(\delta)_d$-comodules.
For simplicity, we will work within the non-quantised setting over an infinite
field. In particular,  the
coordinate variables $c_{ij}$ commute with each other.

We will need some additional notation.
We denote by $\k[T_n]$ the coordinate algebra of the subgroup of diagonal
matrices in $\gl_{n}(\k)$. The canonical projection $\pi\colon A(\delta)_d
\to \k[T_n]$  is defined by
\begin{equation*}
\pi(c_{ij}) = \begin{cases}
c_{ii},& i=j\\
0, & \mbox{otherwise.}
\end{cases}
\end{equation*}
It is straightforward to verify that $\pi$ is a homomorphism of coalgebras.
Therefore, every $A(\delta)_d$-comodule $(M,\rho)$ can be considered a
$\k[T_n]$-comodule with the coaction given by
\begin{equation*}
\rho_T(x ) := (\id \otimes \pi) (\rho(x)),
\end{equation*}
for all $x\in M$.
For every $a\in \Z^n$ we define the one-dimensional $\k[T_n]$-comodule
$\k_a$ by
\begin{equation*}
\rho(1) := 1 \otimes c_{11}^{a_1} \dots c_{nn}^{a_n}.
\end{equation*}
It is well-known that every finite dimensional indecomposable comodule over $\k[T_n]$ is
isomorphic to $\k_a$ for some $a\in \Z^n$.
Given a finite dimensional  $A(\delta)_d$-comodule $M$, we can write
$M = \bigoplus_{a} M_a$, where each $M_a$ is the $\k[T_n]$-
submodule of $M$ satisfying
\begin{equation*}
\rho_T (x) = x \otimes c_{11}^{a_1}\dots c_{nn}^{a_n}
\end{equation*}
for all  $x\in M_a$.
The subspaces  $M_a$ of $M$ are called   \emph{weight subspaces} of $M$. We will say that the elements of $M_{a}$ have \emph{weight $a$}.

Fix $i\in \n$.
We will write $\k[G_i]$ for the coordinate algebra of the Levi subgroup
\begin{equation*}
G_i := \gl_{1}(\k)^{ (i-1)} \times \gl_2(\k) \times
\gl_{1}(\k)^{(n-i-1)}.
\end{equation*}
Thus $\k[G_i]$ is the localization of $\k[c_{11},c_{22},\dots, c_{nn},
c_{i,i+1},c_{i+1,i}]$ with respect to
\begin{equation*}
c_{11} \dots c_{i-1,i-1} (c_{ii} c_{i+1,i+1} - c_{i,i+1}c_{i+1,i})
c_{i+2,i+2}\dots c_{nn}.
\end{equation*}
Note that $A[i]_d$ is the  coordinate algebra of the corresponding parabolic
subgroup in $\gl_n(\k)$. Since the Levi subgroup $G_i$ is a quotient of the
corresponding parabolic subgroup we get a well defined
homomorphism of Hopf algebras
\begin{equation*}
\zeta_i \colon \k[G_i] \to A[i]_d
\end{equation*}
determined by
\begin{equation*}
\zeta_i(c_{kl}) = c_{kl},
\end{equation*}
where $k=l \in \n$ or $\left\{ k,l \right\} = \left\{ i,i+1 \right\}$.
Thus every $\k[G_i]$-comodule can be considered as an $A[i]_d$-comodule via
$\zeta_i$.

For every composition $\mu= (\mu_1, \dots , \mu_n)$  such that
$\mu_i = \mu_{i+1}$, we denote by $\k_\mu$ the one-dimensional $A[i]_d$-comodule  with   coaction
given by
\begin{equation*}
\rho(1) = 1 \otimes c_{11}^{\mu_1}\dots c_{i-1,i-1}^{\mu_{i-1}}
(c_{ii}c_{i+1,i+1}-c_{i,i+1}c_{i+1,i})^{\mu_i} c_{i+2,i+2}^{\mu_{i+2}} \dots
c_{nn}^{\mu_n}.
\end{equation*}

From \cite[ Section 3]{donkin1}, we know  that $\k[c_{ii},c_{i,i+1}]$ is a $\k[G_i]$-subcomodule of the regular
$\k[G_i]$-comodule $\k[G_i]$.
For a natural number $m$, we denote by $Y_{i,m}$ the $m$th homogeneous component
of $\k[c_{ii}, c_{i,i+1}]$. Then $Y_{i,m}$ is a $\k[G_i]$-subcomodule of the
$\k[G_i]$-comodule $\k[c_{ii}, c_{i,i+1}]$.
We write $\ty_{i,m}$ for $Y_{i,m}$ considered as $A[i]_d$-comodule via
$\zeta_i$.

It follows from Lemma~3.1 and Lemma~2.12 in \cite{donkin1} that
\begin{enumerate}[1)]
\item \emph{If $\lambda= (\lambda_1, \dots, \lambda_n)$ is such that $\lambda_i - \lambda_{i+1}=m\ge
0$, then $\pi_i^\circ \k_\lambda \cong \k_\mu \otimes \ty_{i,m}$, where
\begin{equation*}
\mu = (\lambda_1,\dots, \lambda_{i-1}, \lambda_{i+1}, \lambda_{i+1},\dots,
\lambda_n)
\end{equation*}
and $R^k \pi_i^\circ\k_\lambda \cong 0$ for $k\ge 1$.}
\item\emph{ If $\lambda_i - \lambda_{i+1} = -1$, then $R^k \pi_i^\circ \k_\lambda\cong
0$ for all $k\ge 0$.}
\end{enumerate}

From now on we fix $n=3$.
Using the above facts we will give an explicit  description of the
$A(\delta)_d$-comodules $F_w \k_{(1,1,0)}$, for all $w\in \hecke{3}$.

As $\ty_{1,0}$ is the trivial $A[1]_d$-comodule, we get that
\begin{equation*}
\pi_1^\circ \k_{(1,1,0)} \cong \k_{(1,1,0)}.
\end{equation*}
Therefore, $F_1 \k_{(1,1,0)}\cong \k_{(1,1,0)}$.
This implies that $F_2 F_1 \k_{(1,1,0)} \cong F_2 \k_{(1,1,0)}$ and
\begin{equation}
\label{eq2}
F_2 F_1 F_2 \k_{(1,1,0)} \cong F_1 F_2 F_1 \k_{(1,1,0)} \cong F_1 F_2
\k_{(1,1,0)}.
\end{equation}
Thus, to know all the $A(\delta)_d$-comodules $F_w \k_{(1,1,0)}$,  we only have to compute $F_2\k_{(1,1,0)}$ and $F_1 F_2 \k_{(1,1,0)}$.

We start by studying $F_2\k_{(1,1,0)}$. For this, consider $\ty_{2,1}$. It has $\k$-basis $\left\{ c_{22},c_{23} \right\}$ and
 $A[2]_d$-comodule structure given by
\begin{equation*}
 \rho(c_{22}) = c_{22}\otimes c_{22} + c_{23}\otimes c_{32},\quad
\rho(c_{23}) = c_{22}\otimes c_{23} + c_{23} \otimes c_{33}.
\end{equation*}
Let us compute the $A[2]_d$-comodule structure on
\begin{equation*}
\pi_2^\circ \k_{(1,1,0)} \cong \k_{(1,0,0)} \otimes \ty_{2,1}.
\end{equation*}
Since $\rho(x)=x\otimes c_{11}$ for $x \in \k_{(1,0,0)}$,
we get in $\k_{(1,0,0)}\otimes \ty_{2,1}$ 
\begin{equation*}
\rho(1\otimes c_{22}) = (1\otimes c_{22}) \otimes c_{11}c_{22} + (1\otimes
c_{23}) \otimes c_{11}c_{32}, \quad
\rho(1\otimes c_{23}) = (1\otimes c_{22})\otimes c_{11} c_{23} + (1\otimes
c_{23}) \otimes c_{11}c_{33}.
\end{equation*}
Therefore, the $A(\delta)_d$-comodule $F_2 \k_{(1,1,0)} = \pi_{2\circ}
\pi_2^\circ \k_{(1,0,0)}$ is two-dimensional, with basis $\{1\otimes c_{22}, 1\otimes
c_{23}\}$ and $A(\delta)_d$-comodule structure given by
\begin{equation}
\label{eq7}
\rho(1\otimes c_{22}) = (1\otimes c_{22}) \otimes c_{11}c_{22} + (1\otimes
c_{23}) \otimes c_{11}c_{32}, \quad
\rho(1\otimes c_{23}) =  (1\otimes
c_{23}) \otimes c_{11}c_{33}.
\end{equation}
It is now easy to determine the weight subspace structure of $F_2\k_{(1,1,0)}$. This structure will be useful to study $F_1F_2\k_{(1,1,0)}$.

From (\ref{eq7}), we get
\begin{equation*}
\rho_T(1\otimes c_{22}) = (1\otimes c_{22}) \otimes c_{11}c_{22}
,\quad
\rho_T(1\otimes c_{23}) =  (1\otimes
c_{23}) \otimes c_{11}c_{33}.
\end{equation*}
This implies that
\begin{equation*}
\left( F_2\k_{(1,1,0)} \right)_{(1,1,0)} = \left\langle 1\otimes c_{22} \right\rangle
,\quad \left( F_2 \k_{(1,1,0)} \right)_{(1,0,1)} = \left\langle 1\otimes
c_{23}
\right\rangle.
\end{equation*}
Moreover, from~\eqref{eq7} it also follows that
$\left\langle 1\otimes c_{23} \right\rangle$ is an $A(\delta)_d$-subcomodule of
$F_2 \k_{(1,1,0)}$ isomorphic to $\k_{(1,0,1)}$. The corresponding quotient has the following
$A(\delta)_d$-coaction
\begin{equation*}
\rho([1\otimes c_{22}]) = [1\otimes c_{22}] \otimes c_{11}c_{22}
\end{equation*}
and so it is isomorphic to $\k_{(1,1,0)}$.
Thus we get a short exact sequence of $A(\delta)_d$-comodules
\begin{equation}
\label{ses1}
0 \to \k_{(1,0,1)} \to F_2 \k_{(1,1,0)} \to \k_{(1,1,0)} \to 0.
\end{equation}

Next we will study the $A(\delta)_d$-comodule structure of $F_1 F_2 \k_{(1,1,0)}$.  For this we will exhibit first its weight subspaces, and then determine the $A(\delta)_d$-coaction on a weight basis of  $F_1 F_2 \k_{(1,1,0)}$. We start by applying $\pi_1^\circ$ to \eqref{ses1}.
As $R^1\pi_1^\circ \k_{(1,0,1)} \cong 0$, we get from the long exact sequence
that
\begin{equation*}
0 \to \pi_1^\circ \k_{(1,0,1)} \to \pi_1^\circ F_2 \k_{(1,1,0)} \to \pi_1^\circ
\k_{(1,1,0)} \to 0
\end{equation*}
is an exact sequence of $A[1]_d$-comodules.
Applying $\pi_{1\circ}$ and taking into account that $F_1 \k_{(1,1,0)}\cong
\k_{(1,1,0)}$
we get the exact sequence
\begin{equation}
\label{ses2}
0 \to F_1 \k_{(1,0,1)} \to F_1F_2 \k_{(1,1,0)} \to \k_{(1,1,0)} \to 0
\end{equation}
of $A(\delta)_d$-comodules.

By a computation similar to the case of $F_2\k_{(1,1,0)}$, we can see that $F_1
\k_{(1,0,1)}$ has  basis $\{1\otimes c_{11}, 1\otimes c_{12}\}$ and $A(\delta)_d$-comodule structure
\begin{equation}
\label{eq3}
\rho(1\otimes c_{11}) = (1\otimes c_{11}) \otimes c_{11}c_{33} + (1\otimes
c_{12}) \otimes c_{21}c_{33},\quad \rho(1\otimes c_{12}) = (1\otimes
c_{12})\otimes c_{22}c_{33}.
\end{equation}
Hence
\begin{equation*}
\rho_T(1\otimes c_{11}) = (1\otimes c_{11}) \otimes c_{11}c_{33} ,\quad \rho_T(1\otimes c_{12}) = (1\otimes
c_{12})\otimes c_{22}c_{33},
\end{equation*}
and therefore
\begin{equation*}
\left( F_1\k_{(1,0,1)} \right)_{(1,0,1)} = \left\langle 1\otimes c_{11}
\right\rangle, \quad \left( F_1\k_{(1,0,1)} \right)_{(0,1,1)} = \left\langle
1\otimes c_{12} \right\rangle.
\end{equation*}
Denote by $u$ the image of $1\otimes c_{12}$ in $F_1F_2\k_{(1,1,0)}$
under the monomorphism in \eqref{ses2}, and by $v$ the image of $1\otimes
c_{11}$ under the same map. Since the monomorphism in \eqref{ses2} is a
homomorphism of $A(\delta)_d$-comodules, from \eqref{eq3}, we get
\begin{equation*}
\rho(u) = u\otimes c_{22}c_{33}, \quad \rho(v) = v\otimes c_{11}c_{33} + u
\otimes c_{21}c_{33}.
\end{equation*}
Note that $u$ has weight $(0,1,1)$ and $v$ has weight $(1,0,1)$

The sequence \eqref{ses2} splits if considered as a sequence of
$\k[T_3]$-comodules. Thus
\begin{equation*}
F_1F_2 \k_{(1,1,0)} \cong \k_{(0,1,1)} \oplus \k_{(1,0,1)} \oplus \k_{(1,1,0)}
\end{equation*}
as $\k[T_3]$-comodules. 
In particular, every weight subspace of $F_1F_2 \k_{(1,1,0)}$ is one-dimensional
and
\begin{equation*}
(F_1F_2 \k_{(1,1,0)})_{(0,1,1)} = \left\langle u \right\rangle, \quad (F_1F_2
\k_{(1,1,0)})_{(1,0,1)} = \left\langle v \right\rangle.
\end{equation*}

From the explicit description \eqref{eq3} of the $A(\delta)_d$-coaction on
$F_1\k_{(1,0,1)}$, we get the short exact sequence
\begin{equation}\label{ses3}
0 \to \k_{(0,1,1)} \to F_1 \k_{(1,0,1)} \to \k_{(1,0,1)} \to 0.
\end{equation}
As $R^k\pi_2^\circ \k_{(1,0,1)}\cong 0$ for all $k\ge 0$, applying $\pi_2^\circ$
to \eqref{ses3}, we get that
\begin{equation*}
\pi_2^\circ F_1 \k_{(1,0,1)} \cong \pi_2^\circ \k_{(0,1,1)} \cong
\k_{(0,1,1)};\quad R^k\pi_2^\circ F_1 \k_{(1,0,1)} \cong R^k \pi_2^\circ
\k_{(0,1,1)} \cong 0,\ k\ge 1.
\end{equation*}
So applying $\pi_2^\circ$   to \eqref{ses2}, we get the short exact sequence of
$A[2]_d$-comodules
\begin{equation*}
0 \to \k_{(0,1,1)} \to \pi_2^{\circ} F_1 F_2 \k_{(1,1,0)} \to
\pi_2^\circ \k_{(1,1,0)} \to 0.
\end{equation*}
As $\pi_{2\circ}$ is exact and $F_2 = \pi_{2\circ} \pi_{2}^\circ$ we obtain the
short exact sequence
\begin{equation}
\label{ses4}
0 \to \k_{(0,1,1)} \to F_2 F_1 F_2 \k_{(1,1,0)} \to
F_2\k_{(1,1,0)} \to 0
\end{equation}
of $A(\delta)_d$-comodules.
In view of \eqref{eq2}, the short exact sequence \eqref{ses4} becomes
\begin{equation}\label{ses5}
0 \to \k_{(0,1,1)} \to  F_1 F_2 \k_{(1,1,0)} \xrightarrow{\theta}
F_2\k_{(1,1,0)} \to 0.
\end{equation}
Let $\bar{v}$ be the image of $v$ in $F_2 \k_{(1,1,0)}$ under the epimorphimsm
$\theta$ in
\eqref{ses5}.
Since $v$ has weight $(1,0,1)$, the same is true for $\bar{v}$. As the weight
subspace $(F_2\k_{(1,1,0)})_{(1,0,1)}$ is one-dimensional and is spanned by $1\otimes
c_{23}$, there is a non-zero $\gamma\in \k$ such that $\bar{v}= \gamma \otimes
c_{23}$.
The epimorphism $\theta$ in \eqref{ses5} induces an isomorphism between the weight subspaces
$(F_1F_2\k_{(1,1,0)})_{(1,1,0)}$ and $(F_2\k_{(1,1,0)})_{(1,1,0)} = \left\langle
1\otimes c_{22} \right\rangle$. Let us denote
by $w$ the element in $(F_1F_2\k_{(1,1,0)})_{(1,1,0)}$ that corresponds to
$\gamma \otimes c_{22}$ under this isomorphism.
Then by \eqref{eq7}
\begin{equation}\label{eq9}
\begin{aligned}
\rho(\theta (w)) &= \gamma \rho(1\otimes c_{22} )  = (\gamma \otimes c_{22})
\otimes c_{11}c_{22} + (\gamma \otimes c_{23}) \otimes c_{11}c_{32}.
\end{aligned}
\end{equation}

We have that $\left( F_1F_2 \k_{(1,1,0)} \right)_{(1,1,0)} = \left\langle w
\right\rangle$ and that $\left\{ u,v,w \right\}$ is a basis of $F_1F_2
\k_{(1,1,0)}$. Therefore, there are unique $h$ and $f$ in $A(\delta)_d$ such that
\begin{equation}
\label{rhow}
 \rho(w) = w \otimes c_{11}c_{22} + v \otimes h + u \otimes f
\end{equation}
and $\pi(h) = \pi(f) = 0$.  Thus
\begin{equation}\label{a;sdl}
\begin{aligned}
 (\theta \otimes \id) \rho(w) & = (\gamma \otimes c_{22}) \otimes
c_{11}c_{22} + (\gamma \otimes c_{23})\otimes h.
\end{aligned}
\end{equation}
Since $(\theta \otimes \id)  \rho = \rho \theta$, compairing
\eqref{eq9} and \eqref{a;sdl}, we get that
$h=  c_{11}c_{32}$. Hence it is left to determine $f$.

As $\k_{(1,1,0)}$ is an $A(\delta;2)$-comodule, we get from the considerations at
the end of Section~\ref{seventh}, that $F_1F_2 \k_{(1,1,0)}$ is an
$A(\delta;2)$-comodule. Therefore $f$ is an element of degree two in $A(\delta)\cong
\k[c_{11},c_{22},c_{33},c_{21},c_{31},c_{32}]$.

From \eqref{rhow}, we get
\begin{align*}
(\rho\otimes \id)\rho(w)& = w \otimes c_{11}c_{22}\otimes c_{11}c_{22} +  v \otimes
(c_{11}c_{32} \otimes c_{11}c_{22}+ c_{11}c_{33}\otimes c_{11}c_{32}  )\\
&\phantom{=} + u
\otimes (f\otimes c_{11}c_{22} +  c_{21}c_{33}\otimes c_{11}c_{32} +
c_{22}c_{33}\otimes f);\\
(\id \otimes \Delta)\rho(w) & = w \otimes c_{11}c_{22}\otimes c_{11}c_{22} +
 v \otimes (c_{11}c_{32} \otimes c_{11}c_{22} + c_{11}c_{33}\otimes
c_{11} c_{32}) \\& \phantom{=} +
u \otimes \Delta(f).
\end{align*}
As $(\rho\otimes \id)  \rho = (\id \otimes \Delta ) \rho$
we obtain that  $f$  satisfies the equation
\begin{equation}
\label{eq1}
\Delta(f) = f\otimes c_{11}c_{22} +  c_{21}c_{33} \otimes
c_{11} c_{32} + c_{22}c_{33} \otimes f.
\end{equation}
Denote by $V$ the subspace of $A(\delta;2)\otimes A(\delta;2)$ spanned by
\begin{equation}
\label{vbasis}
 \left\{ c_{ij}c_{kl}\otimes c_{js}c_{lt}\ \middle|\  i\ge j\ge s,\ k\ge l\ge t
\right\}.
\end{equation}
From the definition of the comultiplication in $A(\delta)$, we get that
$\Delta (A(\delta;2)) \subset V$.

Suppose $c_{ij}c_{kl}$ has  non-zero coefficient  in the expansion of
$f$ with respect to the monomial basis of $A(\delta;2)$. Note that $c_{ij}c_{kl}\not=
c_{11}c_{22}$, $c_{ij}c_{kl}\not= c_{22}c_{33}$ as $\pi(f)=0$. Then
from~\eqref{eq1}, we see that   $c_{ij}c_{kl} \otimes
c_{11}c_{22}$ and $c_{22}c_{33}\otimes c_{ij}c_{kl}$ have non-zero coefficients
in the expansion of $\Delta(f)\in \Delta (A(\delta;2))\subset V$ with respect to
the basis \eqref{vbasis} of $V$.
Thus $\left\{ j,l \right\}=\{1,2\}$ and $\left\{ i,k \right\}=\{2,3\}$. Therefore
the only basis elements of $A(\delta;2)$ that can have  non-zero coefficients in
the expansion of
$f$ are $c_{21}c_{32}$ and $c_{22}c_{31}$.
Direct computation now shows that
the only linear combination of $c_{21}c_{32}$ and $c_{22}c_{31}$ that satisfy \eqref{eq1} is
\begin{equation*}
f = c_{21}c_{32} - c_{22}c_{31}.
\end{equation*}
Hence we get a full description of $A(\delta)_d$-comodule structure on $F_1F_2
\k_{(1,1,0)}$:
\begin{equation*}
\rho(u) = u\otimes c_{22}c_{33},\ \rho(v) = v\otimes c_{11}c_{33} + u\otimes
c_{21}c_{33},\ \rho(w) = w\otimes c_{11}c_{22} + v\otimes c_{11}c_{32} +
u\otimes (c_{21}c_{32} - c_{22}c_{31}).
\end{equation*}

\bibliography{hecke}
\bibliographystyle{amsplain}

\end{document}